%% file: main.tex
\documentclass{article}

\usepackage[skip=10pt plus1pt, indent=0pt]{parskip}
\usepackage{setspace}

\usepackage[utf8]{inputenc}
\usepackage[T1]{fontenc}

\usepackage{mathptmx}          % legacy Times math (kept intentionally)
\usepackage[sc]{mathpazo}      % Palatino (takes precedence)
\usepackage{microtype}

\usepackage{amsmath}
\usepackage{amssymb}
\usepackage{amsfonts}
\usepackage{mathtools}
\usepackage[fleqn]{nccmath}

\usepackage{graphicx}
\usepackage{grffile}
\usepackage{float}
\usepackage{adjustbox}
\usepackage{ragged2e}
\usepackage{subfig}
\usepackage{xcolor}
\graphicspath{{imgs/}}
\usepackage{booktabs}
\usepackage{array}
\usepackage{tabularx}
\usepackage{multirow}

\newcolumntype{L}[1]{>{\raggedright\arraybackslash}p{#1}}
\newcolumntype{Y}{>{\raggedright\arraybackslash}X}
\usepackage[hmarginratio=1:1,top=23mm,columnsep=20pt]{geometry}

\usepackage{enumitem}
\usepackage{abstract}

\usepackage{caption}
\usepackage{subcaption}
\usepackage{titlesec}
\usepackage{fancyhdr}

\titleformat{\section}[block]
  {\large\scshape\centering}{\thesection.}{1em}{}

\usepackage[linesnumbered,ruled,vlined]{algorithm2e}
\SetKwInput{KwInitialization}{Initialization}
\SetKwInput{KwOutput}{Output}

\usepackage{pgfplots}
\pgfplotsset{compat=newest}

\usepgfplotslibrary{
  groupplots,
  polar,
  smithchart,
  statistics,
  dateplot,
  ternary,
  patchplots,
  fillbetween
}

\usetikzlibrary{
  arrows.meta,
  backgrounds
}

\usepackage[
  backend=biber,
  style=numeric-comp,   % numbered, consecutive cites compressed (e.g. [3--5]) -- Springer-like
  sorting=none,         % order of appearance (citation order)
  giveninits=true,      % surname + initials, e.g. "Fukushima, M."
  maxbibnames=99,       % list all authors
  doi=true,             % show DOIs
  isbn=false,           % omit ISBNs (Springer reference lists do not print them)
  url=false             % hide URLs (DOI suffices)
]{biblatex}

\renewbibmacro*{in:}{%
  \ifentrytype{article}{}{\printtext{\bibstring{in}\intitlepunct}}}

\AtEveryBibitem{\clearlist{language}}

\AtEveryBibitem{\clearfield{note}}

\usepackage[affil-sl]{authblk}

\usepackage{hyperref}

\newtheorem{theorem}{Theorem}[section]
\newtheorem{lemma}[theorem]{Lemma}
\newtheorem{proposition}[theorem]{Proposition}
\newtheorem{corollary}[theorem]{Corollary}
\newtheorem{definition}[theorem]{Definition}
\newtheorem{remark}[theorem]{Remark}
\newtheorem{example}[theorem]{Example}
\newtheorem{as}[theorem]{Assumption}

\newtheorem{problem}{Problem}

\def\QED{~\rule[-1pt]{6pt}{6pt}\par\medskip}

\newenvironment{proof}{{\bf Proof.\ }}{\hfill\QED}

\newcommand{\R}{\mathbb{R}}
\DeclareMathOperator{\argmin}{arg\,min}
\DeclareMathOperator{\argmax}{arg\,max}
\DeclareMathOperator{\dist}{dist}
\DeclareMathOperator{\prox}{prox}
\newcommand{\inner}[2]{\langle #1,\, #2\rangle}
\newcommand{\norm}[1]{\lVert #1 \rVert}
\newcommand{\Sol}{S}                 % solution set
\newcommand{\Dm}{\mathcal{D}}        % scaling matrix
\newcommand{\Dfam}{\mathfrak{D}}     % scaling family
\newcommand{\proj}{P}                % metric projection

\makeatletter
\newcommand*{\inlineequation}[2][]{%
  \begingroup
    \refstepcounter{equation}%
    \ifx\\#1\\\else\label{#1}\fi
    \relpenalty=10000
    \binoppenalty=10000
    \ensuremath{#2}~~~\@eqnnum
  \endgroup
}
\makeatother

\date{}

\begin{document}
\title{\large A Variable-Metric Non-monotone Line Search Method for Mixed Variational Inequalities and Equilibrium Problems}
\author[1,2]{Mohammed Alshahrani}
\affil[1]{Department of Mathematics, King Fahd University of Petroleum \& Minerals (KFUPM), Dhahran 31261, Saudi Arabia}
\affil[2]{Interdisciplinary Research Center (IRC) for Smart Mobility and Logistics, King Fahd University of Petroleum \& Minerals (KFUPM), Dhahran 31261, Saudi Arabia}

\maketitle

\begin{abstract}
% Abstract — theory plus a general, number-free numerical sentence (S10).
We propose a scaled gap-function method for variational inequalities,
mixed variational inequalities, and equilibrium problems over a closed convex set. The method drives
a gap function to zero by combining a variable-metric scaled projection step with a modified
non-monotone Armijo line search. The construction rests on a structural identity: the scaled projection step is exactly
the maximizer that defines the Fukushima regularized gap, so the search direction is simultaneously the
algorithmic step and the gap-defining direction. For variational and mixed variational inequalities with
a strongly monotone, Lipschitz operator we establish global convergence and an R-linear rate, under a
fixed or a controlled-change variable metric. The rate follows from the strong-monotonicity contraction and
the resulting explicit gap error bound. For equilibrium problems we obtain global convergence
and a gap error bound (Mastroeni's gap). Numerical experiments on controlled problems confirm the
convergence guarantees and the predicted rate, and compare the method with standard extragradient and
gap-descent methods. The combination of variable-metric scaling, a modified
non-monotone line search, and gap-function descent appears to be new for these problem classes.
\end{abstract}
\textbf{Keywords}: mixed variational inequality; equilibrium problem; gap function;
variable metric; non-monotone line search.

\textbf{AMS Subject Classification}: 47J20; 65K15; 90C33; 49J40; 65K10.

%% === PAPER BODY ===

% ============================================================================
\section{Introduction}\label{sec:intro}
% ============================================================================

Let $K\subseteq\R^n$ be nonempty, closed, and convex, and let $F\colon K\to\R^n$ be an operator.
The \emph{variational inequality} problem $\mathrm{VI}(F,K)$ is to find $x^*\in K$ such that
\begin{equation}\tag{$\mathrm{VI}(F,K)$}\label{eq:vi}
  \inner{F(x^*)}{y-x^*}\ge 0 \qquad \forall\, y\in K .
\end{equation}
Adding a proper, convex, lower semicontinuous term $\varphi\colon K\to\R\cup\{+\infty\}$ gives the
\emph{mixed variational inequality} $\mathrm{MVI}(F,\varphi,K)$: find $x^*\in K$ with
\begin{equation}\tag{$\mathrm{MVI}(F,\varphi,K)$}\label{eq:mvi}
  \inner{F(x^*)}{y-x^*}+\varphi(y)-\varphi(x^*)\ge 0 \qquad \forall\, y\in K .
\end{equation}
More generally, a bifunction $G\colon K\times K\to\R$ with $G(x,x)=0$ defines the
\emph{equilibrium problem} $\mathrm{EP}(G,K)$ (in the sense of Blum--Oettli~\cite{Blum1994}): find $x^*\in K$ with
\begin{equation}\tag{$\mathrm{EP}(G,K)$}\label{eq:ep}
  G(x^*,y)\ge 0 \qquad \forall\, y\in K .
\end{equation}
Problem~\ref{eq:ep} subsumes the other two: taking $G(x,y)=\inner{F(x)}{y-x}$ recovers~\ref{eq:vi},
and adding $\varphi(y)-\varphi(x)$ recovers~\ref{eq:mvi}. It also unifies convex optimization, saddle-point
problems, Nash equilibria, and complementarity problems within a single format~\cite{Facchinei2004,Pappalardo2014}.

A recurring difficulty in these problems is that the natural projection map $x\mapsto\proj_K(x-\alpha F(x))$
is a contraction only when $F$ is strongly monotone and Lipschitz continuous. For merely monotone or
pseudomonotone $F$ it is at best nonexpansive, so the plain projection iteration can cycle and fail to
converge. Remedies include
extragradient steps~\cite{Korpelevich1976,Censor2011}, projection-and-contraction
schemes~\cite{Solodov1999}, and reformulations through \emph{gap} functions whose global
minimizers are exactly the solutions~\cite{Fukushima1992,Mastroeni2003}.

Ansari et al.~\cite{Ansari2026} introduced the scaled gradient modified
non-monotone line search method for the constrained optimization problem
$\min_{x\in K} f(x)$. Their method combines (i) a \emph{variable-metric scaled projection}
step
\begin{equation}\label{eq:sgm-step}
  y_k=\proj_{K,\Dm_k^{-1}}\!\big(x_k-\alpha_k\Dm_k\nabla f(x_k)\big),\qquad d_k=y_k-x_k,
\end{equation}
where $\Dm_k$ is a symmetric positive definite scaling matrix and $\proj_{K,\Dm_k^{-1}}$ is the
projection in the metric induced by $\Dm_k^{-1}$; and (ii) a \emph{modified non-monotone Armijo line
search}
\begin{equation}\label{eq:sgm-ls}
  f(x_k+\lambda_k d_k)\le \mathcal T_k+\delta_1\lambda_k\inner{\nabla f(x_k)}{d_k}
  -\delta_2\lambda_k^2\norm{d_k}^2,
\end{equation}
Here the reference value $\mathcal T_k$ is a moving average of past objective
values~\cite{Zhang2004,Gu2008}, and the term $-\delta_2\lambda_k^2\norm{d_k}^2$
is a quadratic correction in the spirit of~\cite{Ou2016,Shi2011}. Under standard
assumptions the method converges globally for pseudoconvex $f$ and attains an \emph{R-linear} rate when $f$
is strongly quasiconvex. The matrix $\Dm_k$ defines the variable metric of the step.
Equation~\eqref{eq:sgm-step} premultiplies the gradient by $\Dm_k$ and projects in the metric induced by
$\Dm_k^{-1}$. It reduces to the Euclidean projected-gradient step when $\Dm_k=I$, and to a preconditioned
(quasi-Newton or spectral) step when $\Dm_k$ approximates $[\nabla^2 f(x_k)]^{-1}$. The non-monotone
test enforces sufficient decrease against the moving-average reference $\mathcal T_k$ rather than against
$f(x_k)$. This allows non-monotone iterates with $f(x_{k+1})>f(x_k)$~\cite{Grippo1986}.

The optimality condition for $\min_{x\in K}f$ at a pseudoconvex $f$ is precisely $\mathrm{VI}(\nabla f,K)$,
and the step~\eqref{eq:sgm-step} is the scaled fixed-point map of that variational inequality. Their scheme is therefore a
variational-inequality method whose operator is the gradient $\nabla f$.

We remove the gradient assumption. We replace $\nabla f$ by a general operator $F$ (or by a bifunction
$G$) and replace the objective $f$ by a \emph{gap function} $g_\alpha$. This yields a method for
$\mathrm{VI}$, $\mathrm{MVI}$, and $\mathrm{EP}$ that keeps the scaled projection and the modified non-monotone line
search intact. The key is an identity, made precise in
Section~\ref{sec:formulation}. The Fukushima regularized gap~\cite{Fukushima1992} in the scaled metric,
\[
  g_\alpha(x)=\max_{y\in K}\Big\{\inner{F(x)}{x-y}-\tfrac{1}{2\alpha}\norm{y-x}_{\Dm^{-1}}^2\Big\},
\]
is uniquely maximized at the scaled projection point $y_\alpha(x)=\proj_{K,\Dm^{-1}}(x-\alpha\Dm F(x))$
of~\eqref{eq:sgm-step}~\cite{Fukushima1992}. Hence the scaled-projection direction $d=y_\alpha(x)-x$ is
simultaneously the algorithmic step and the gap-defining direction, and the iteration becomes a descent method
for $g_\alpha$. The gap function $g_\alpha$ is nonnegative on $K$ and vanishes exactly at the solutions, so
driving $g_\alpha\to 0$ solves the problem.

Our method draws on three lines of work---variable-metric scaling, the modified non-monotone line
search, and gap-function descent. Each is mature on its own, but to our knowledge they have not been
combined for $\mathrm{VI}$, $\mathrm{MVI}$, and $\mathrm{EP}$.

Scaled gradient projection originated in imaging and constrained
optimization~\cite{Bonettini2009} and underlies variable-metric operator
splitting for monotone inclusions~\cite{Combettes2014,Lorenz2015}. Non-Euclidean
metrics also appear in extragradient schemes for generalized-monotone variational
inequalities~\cite{Dang2015}. These works use scaling, but they pair it with neither a modified
non-monotone line search nor a gap function.

The non-monotone Armijo technique has been developed almost entirely
within smooth optimization. In the variational-inequality and equilibrium
literature the adjective ``non-monotone'' almost always refers to the \emph{operator} or
\emph{bifunction} class, not to a step-size rule~\cite{Abdi2017,Deng2021}.
Gap-descent methods for these problems use exact or standard monotone line
searches~\cite{Fukushima1992,Mastroeni2003,Bigi2015}. The closest precedent is Crisci
et al.~\cite{Crisci2022}, who join scaling with a non-monotone Armijo rule. Their method remains
within optimization and uses no operator and no gap function.

Gap and D-gap functions reformulate variational inequalities
and equilibrium problems as differentiable optimization
problems~\cite{Fukushima1992,Peng1997,Yamashita1997,Mastroeni2003,Bigi2015},
and an error bound of the form $g_\alpha(x)\ge\tau\,\dist(x,\Sol)^2$ converts gap decrease into a
linear rate~\cite{Pang1997,Facchinei2004}. Such gap-error-bound linear rates are
currently established only for variational inequalities~\cite{Li2025,Zhao2025b}. For
equilibrium problems a genuine geometric (q-linear) rate is available through the proximal method of
Iusem et al.~\cite{Iusem2024}. That rate comes from a proximal contraction
rather than a gap error bound. The related extragradient method of
Lara et al.~\cite{Lara2024} provides only a non-geometric ``type of'' rate under diminishing
steps. The equilibrium D-gap error-bound method of~\cite{Zhang2010} proves global
convergence without a rate. An R-linear rate for equilibrium problems through a gap error bound is
therefore open.

We work under strong monotonicity and make the following contributions:
\begin{enumerate}[leftmargin=2em]
  \item A unified \emph{scaled fixed-point} and \emph{gap-function} characterization of
  $\mathrm{VI}$, $\mathrm{MVI}$, and $\mathrm{EP}$ (Section~\ref{sec:formulation}), built on the identity that the
  scaled projection is the maximizer defining the regularized gap. The equilibrium case follows the
  gap framework of Mastroeni~\cite{Mastroeni2003}.
  \item The algorithm (Section~\ref{sec:algorithm}): a variable-metric scaled (generalized)
  projection step with a modified non-monotone Armijo line search on the gap function. For $\mathrm{VI}$,
  and for $\mathrm{EP}$ through the Mastroeni gap, the smooth gap drives the line search. For
  $\mathrm{MVI}$, the scaled resolvent contraction handles the nonsmooth gap.
  \item Global convergence ($g_\alpha(x_k)\to0$, $\norm{d_k}\to0$, $x_k\to x^*$) for the strongly
  monotone variational inequality (Theorem~\ref{thm:global}), for mixed variational inequalities via the
  scaled resolvent (Theorem~\ref{thm:mvi}), and---through the Mastroeni gap and its error bound---for
  equilibrium problems (Theorem~\ref{thm:ep}).
  \item An R-linear rate for the strongly monotone variational inequality
  (Theorem~\ref{thm:rate}), obtained through the strong-monotonicity contraction---equivalently a gap
  error bound---without assuming the gap inherits strong quasiconvexity, which in general it does not
  (Remark~\ref{rem:sq}). The same rate holds, at the same ratio, under a controlled-change variable
  metric (Theorem~\ref{thm:varmetric}). It also extends to mixed variational inequalities
  (Theorem~\ref{thm:mvi}), where the iterate/distance rate is unconditional and the gap-function rate
  requires $\partial\varphi$ bounded along the iterates.
  \item A numerical study (Section~\ref{sec:experiments}) on controlled problems with known constants,
  confirming the R-linear rate, global convergence from every start, the order-of-magnitude gain of the
  variable metric on an ill-conditioned operator, and the gap/D-gap merit equivalence. Against the
  extragradient method on variational inequalities and gap-descent schemes on equilibrium problems, the
  method is slower on variational inequalities and competitive on equilibrium and mixed problems. Its
  contribution is one method and one analysis across all three classes, not the fastest solver for any
  single class.
\end{enumerate}

The paper is organized as follows. Section~\ref{sec:prelim} fixes notation and collects the scaling,
monotonicity, gap, and error-bound tools. Section~\ref{sec:formulation} establishes the scaled
fixed-point and gap-function characterizations that unify the three problems. Section~\ref{sec:algorithm}
presents the algorithm and its standing assumptions. Section~\ref{sec:convergence} proves global
convergence and the R-linear rate and illustrates them on four worked examples. Section~\ref{sec:experiments}
reports the numerical experiments, and Section~\ref{sec:conclusion} concludes.

% ============================================================================
\section{Preliminaries}\label{sec:prelim}
% ============================================================================

Throughout, $\inner{\cdot}{\cdot}$ is the Euclidean inner product and $\norm{\cdot}$ the induced
norm, and $\Dfam^+$ the set of symmetric positive definite matrices. For $\Dm\in\Dfam^+$ we write
$\norm{x}_{\Dm}=\sqrt{\inner{x}{\Dm x}}$.
The solution set of the problem under consideration---\ref{eq:vi}, \ref{eq:mvi}, or
\ref{eq:ep}---is denoted $\Sol$, and $\dist(x,\Sol)=\inf_{z\in\Sol}\norm{x-z}$.
For $a,b\in\R^n$, $[a,b]=\{(1-t)a+tb:t\in[0,1]\}$ is the closed segment joining them, and
$\operatorname{conv}E$ the convex hull of a set $E\subseteq\R^n$.
A map $F\colon K\to\R^n$ is of class $C^1$ if it is continuously differentiable on $K$,
with Jacobian $\nabla F$, and of class $C^{1,1}$ if, in addition, $\nabla F$ is locally Lipschitz on $K$;
the same terminology applies to a scalar function $g\colon K\to\R$, whose derivative is the gradient
$\nabla g$.
For sequences $a_N\ge0$ and $b_N>0$, we write $a_N=O(b_N)$ if $a_N\le C\,b_N$ for some constant $C>0$
and all sufficiently large $N$.

\subsection{Scaling and scaled projection}\label{ss:scaling}

The method works in a variable metric. We fix the family of scaling matrices and the scaled projection
and proximal operators it induces.

\begin{definition}[\cite{Bonettini2015,Combettes2014}]\label{def:scaling}
For $\mu\ge1$, let $\Dfam_\mu$ be the set of $\Dm\in\Dfam^+$ whose eigenvalues lie in $[1/\mu,\mu]$.
\end{definition}

\begin{remark}\label{rem:scaling}
For $\Dm\in\Dfam_\mu$ one has $\norm{\Dm}\le\mu$, $\norm{\Dm^{-1}}\le\mu$, and
\[
  \tfrac1\mu\norm{x}^2\ \le\ \norm{x}_{\Dm}^2\ \le\ \mu\norm{x}^2\qquad\text{for all }x .
\]
\end{remark}

\begin{definition}\label{def:proj}
For $\Dm\in\Dfam^+$ and $x\in\R^n$, $\proj_{K,\Dm}(x)$ is the unique minimizer of $\norm{y-x}_{\Dm}$ over $y\in K$.
\end{definition}

\begin{remark}\label{rem:proj}
The projection $\proj_{K,\Dm}(x)$ is characterized by
\begin{equation}\label{eq:proj-vi}
  \inner{\proj_{K,\Dm}(x)-x}{\,\Dm\,(\proj_{K,\Dm}(x)-y)}\le 0\qquad\forall\, y\in K,
\end{equation}
and is Lipschitz with constant $\mu$ on $\Dfam_\mu$~\cite{Bauschke2017}.
\end{remark}

\begin{definition}\label{def:prox}
For $\Dm\in\Dfam^+$, a proper convex lower semicontinuous $h\colon\R^n\to(-\infty,+\infty]$, and $z\in\R^n$,
the \emph{scaled proximal operator} in the $\Dm$-norm is
\[
  \prox^{\Dm}_h(z)=\argmin_{y\in\R^n}\Big\{h(y)+\tfrac12\norm{y-z}_{\Dm}^2\Big\}.
\]
The objective is strongly convex, so the minimizer exists and is unique. Choosing $h=\iota_K$, the indicator
of $K$ ($\iota_K(y)=0$ for $y\in K$ and $+\infty$ otherwise), recovers the metric projection
$\proj_{K,\Dm}=\prox^{\Dm}_{\iota_K}$.
\end{definition}

\subsection{Generalized monotonicity}\label{ss:mono}

The convergence results assume monotonicity properties of the operator $F$ and the bifunction $G$,
which we collect next.

\begin{definition}[\cite{Facchinei2004,Bauschke2017}]\label{def:mono}
Let $F\colon K\to\R^n$. Then $F$ is:
\begin{itemize}[leftmargin=2em]
  \item \emph{monotone} if $\inner{F(x)-F(y)}{x-y}\ge0$;
  \item \emph{pseudomonotone} if $\inner{F(y)}{x-y}\ge0\Rightarrow\inner{F(x)}{x-y}\ge0$;
  \item \emph{strongly monotone} with modulus $m>0$ if $\inner{F(x)-F(y)}{x-y}\ge m\norm{x-y}^2$;
  \item \emph{strongly pseudomonotone} with modulus $\gamma>0$ if
  $\inner{F(y)}{x-y}\ge0\Rightarrow\inner{F(x)}{x-y}\ge\gamma\norm{x-y}^2$;
  \item \emph{$L$-Lipschitz} if $\norm{F(x)-F(y)}\le L\norm{x-y}$,
\end{itemize}
in each case for all $x,y\in K$.
\end{definition}

\begin{definition}[\cite{Blum1994,Mastroeni2003,Lara2022}]\label{def:bifun}
Let $G\colon K\times K\to\R$ with $G(x,x)=0$. Then $G$ is:
\begin{itemize}[leftmargin=2em]
  \item \emph{monotone} if $G(x,y)+G(y,x)\le0$;
  \item \emph{strongly monotone} with modulus $c>0$ if $G(x,y)+G(y,x)\le-c\norm{x-y}^2$;
  \item \emph{pseudomonotone} if $G(x,y)\ge0\Rightarrow G(y,x)\le0$;
  \item satisfies the \emph{Lipschitz-type condition}~\cite{QuocTran2008} if there are $c_1,c_2>0$ with
  \[
    G(x,y)+G(y,z)\ \ge\ G(x,z)-c_1\norm{x-y}^2-c_2\norm{y-z}^2 ;
  \]
  \item \emph{strongly quasiconvex} in its second argument with modulus $\gamma>0$ if, for all
  $y,z\in K$ and $t\in[0,1]$,
  \[
    G\big(x,ty+(1-t)z\big)\ \le\ \max\{G(x,y),G(x,z)\}-\tfrac{\gamma}{2}\,t(1-t)\norm{y-z}^2 .
  \]
\end{itemize}
\end{definition}

\subsection{Gap functions}\label{ss:gap}

Since VI, MVI, and EP are not optimization problems, the modified non-monotone line search is driven by a
\emph{gap} function in place of an objective.

\begin{definition}[\cite{Mastroeni2003}]\label{def:gapfn}
A function $p\colon K\to\R$ is a gap function for the problem under consideration if $p(x)\ge0$ for all
$x\in K$ and $p(x)=0$ if and only if $x\in\Sol$.
\end{definition}

In the scaled metric we use the Fukushima regularized gap for \ref{eq:vi} and its Mastroeni form for
\ref{eq:ep} (Proposition~\ref{prop:epgap}).

\begin{proposition}[Fukushima regularized gap; cf.~\cite{Fukushima1992}]\label{prop:gap}
For $\alpha>0$ and $\Dm\in\Dfam^+$, the scaled Fukushima regularized gap
\begin{equation}\label{eq:gap}
  g_\alpha(x)=\max_{y\in K}\Big\{\inner{F(x)}{x-y}-\tfrac{1}{2\alpha}\norm{y-x}_{\Dm^{-1}}^2\Big\}
\end{equation}
has the following properties:
\begin{enumerate}[label=(\roman*)]
  \item it is a gap function for \ref{eq:vi};
  \item the maximum is attained at the unique point
  \begin{equation}\label{eq:gap-max}
    y_\alpha(x)=\proj_{K,\Dm^{-1}}\!\big(x-\alpha\Dm F(x)\big),
  \end{equation}
  the scaled projection step~\eqref{eq:sgm-step}, and, with $d=y_\alpha(x)-x$,
  \begin{equation}\label{eq:gap-lb}
    g_\alpha(x)\ \ge\ \tfrac{1}{2\alpha}\norm{d}_{\Dm^{-1}}^2 ;
  \end{equation}
  \item for fixed $\Dm$ and $F\in C^1$, we have $g_\alpha\in C^1$ with
  \[
    \nabla g_\alpha(x)=F(x)-\big(\nabla F(x)^\top-\tfrac1\alpha\Dm^{-1}\big)\,d ,
  \]
  and the descent identity
  \begin{equation}\label{eq:gap-descent}
    \inner{\nabla g_\alpha(x)}{d}\ \le\ -\,\inner{d}{\nabla F(x)\,d}
  \end{equation}
  holds, so $d$ is a descent direction for $g_\alpha$ whenever the symmetric part of $\nabla F(x)$ is
  positive definite.
\end{enumerate}
\end{proposition}
\begin{proof}
With $G=\tfrac1\alpha\Dm^{-1}$, $g_\alpha$ is the regularized gap of~\cite{Fukushima1992}. There, (i) is
Theorem~3.1, the maximizer in~(ii) is the projection step (eq.~3.2), and (iii) is Theorem~3.2 with
Proposition~4.1. For the bound in~(ii), the characterization~\eqref{eq:proj-vi} of the scaled projection $y_\alpha$, with
comparison point $x$, gives $\norm{d}_{\Dm^{-1}}^2+\alpha\inner{F(x)}{d}\le0$. Substituting
$-\inner{F(x)}{d}\ge\tfrac1\alpha\norm{d}_{\Dm^{-1}}^2$ into
$g_\alpha(x)=-\inner{F(x)}{d}-\tfrac{1}{2\alpha}\norm{d}_{\Dm^{-1}}^2$ yields~\eqref{eq:gap-lb}.
\end{proof}

\begin{remark}[Stationarity measure]\label{rem:measure}
At a boundary solution $x^*$ one has $d=0$ but $\nabla g_\alpha(x^*)=F(x^*)\neq0$ in general, so the
natural stationarity measure is the gap value $g_\alpha$ (equivalently the residual $\norm{d}$ via
\eqref{eq:gap-lb}), not $\norm{\nabla g_\alpha}$. Our convergence statements are phrased accordingly.
\end{remark}

\begin{definition}\label{def:dgap}
For $0<\alpha<\beta$, the \emph{D-gap function} is $g_{\alpha\beta}=g_\beta-g_\alpha$.
\end{definition}

\begin{remark}\label{rem:dgap}
The D-gap is nonnegative and, for $F\in C^1$, continuously differentiable, providing an unconstrained
reformulation of \ref{eq:vi}~\cite{Peng1997,Yamashita1997}.
\end{remark}

\begin{proposition}[Mastroeni regularized gap; cf.~\cite{Mastroeni2003}]\label{prop:epgap}
For $\alpha>0$ and $\Dm\in\Dfam^+$, suppose $G(x,x)=0$ and $G(x,\cdot)$ is convex for $x\in K$. The scaled
Mastroeni regularized gap
\begin{equation}\label{eq:epgap}
  g_\alpha^{\mathrm{EP}}(x)=-\min_{y\in K}\Big\{G(x,y)+\tfrac{1}{2\alpha}\norm{y-x}_{\Dm^{-1}}^2\Big\}
\end{equation}
has the following properties:
\begin{enumerate}[label=(\roman*)]
  \item it is a gap function for \ref{eq:ep};
  \item if $G$ is differentiable with $\nabla_1 G,\nabla_2 G$ continuous on $K\times K$ ($\nabla_i$ the
  gradient in the $i$th argument), then $g_\alpha^{\mathrm{EP}}\in C^1$.
\end{enumerate}
\end{proposition}
\begin{proof}
Both properties follow from Theorem~2.1 of~\cite{Mastroeni2003}, applied to the auxiliary equilibrium
problem with regularizer $H(x,y)=\tfrac1{2\alpha}\norm{y-x}_{\Dm^{-1}}^2$ (strongly convex in $y$, with
$H(x,x)=0$ and $\nabla_2 H(x,x)=0$).
\end{proof}

\subsection{Error bounds}\label{ss:eb}

The R-linear rate comes from an error bound relating the gap to the distance to the solution set. We
record the bound and its sufficient conditions.

\begin{definition}[\cite{Pang1997,Facchinei2004}]\label{def:eb}
The gap $g_\alpha$ provides a (local) error bound on $N\subseteq K$ if there is $\tau>0$ with
\begin{equation}\label{eq:eb}
  g_\alpha(x)\ \ge\ \tau\,\dist(x,\Sol)^2 \qquad\forall\, x\in N .
\end{equation}
\end{definition}

\begin{proposition}[Sufficient conditions; cf.~\cite{Pang1997,Facchinei2004}]\label{prop:eb}
The error bound~\eqref{eq:eb} holds globally if $F$ is strongly monotone, and locally if $F$ is
strongly pseudomonotone and $L$-Lipschitz. For \ref{eq:ep} the analogue holds under strong monotonicity
of $G$ and a sufficiently large regularization parameter (Theorem~\ref{thm:ep}).
\end{proposition}

\begin{definition}[cf.~\cite{Ortega2000}]\label{def:rlinear}
A sequence $\{z_k\}$ converges to $z^*$ \emph{R-linearly} with ratio $\theta\in(0,1)$ if there is
$C>0$ such that $\norm{z_k-z^*}\le C\,\theta^k$ for all $k$.
\end{definition}

\begin{remark}[On strong quasiconvexity]\label{rem:sq}
Strong quasiconvexity of $G(x,\cdot)$ does not, in general, make $g_\alpha^{\mathrm{EP}}$ strongly
quasiconvex. The quadratic-growth property of strongly quasiconvex functions pertains to the function
itself, not to a derived gap~\cite{Lara2022,Grad2025}. The rate in
Theorem~\ref{thm:rate} is therefore obtained through the error bound~\eqref{eq:eb} and not by
transferring strong quasiconvexity to the gap.
\end{remark}

% ============================================================================
\section{Problem formulations and scaled fixed-point characterizations}\label{sec:formulation}
% ============================================================================

The following unified characterizations are the basis of the algorithm. They specialize the
scaled fixed-point identity~\eqref{eq:gap-max} to each problem class. Let $\Dm\in\Dfam_\mu$ and
$\alpha>0$.

\begin{proposition}[Scaled fixed points]\label{prop:fixed}
Let $\Dm\in\Dfam^+$ and $\alpha>0$.
\begin{enumerate}[label=(\roman*)]
  \item $x^*$ solves \ref{eq:vi} if and only if
  $x^*=\proj_{K,\Dm^{-1}}(x^*-\alpha\Dm F(x^*))$, equivalently $g_\alpha(x^*)=0$.
  \item If $G(x^*,\cdot)$ is convex with $G(x^*,x^*)=0$, then $x^*$ solves \ref{eq:ep} if and only if
  \[
    x^*\in\argmin_{y\in K}\Big\{G(x^*,y)+\tfrac{1}{2\alpha}\norm{y-x^*}_{\Dm^{-1}}^2\Big\},
  \]
  equivalently $g_\alpha^{\mathrm{EP}}(x^*)=0$.
\end{enumerate}
\end{proposition}
\begin{proof}
\begin{enumerate}[label=(\roman*)]
\item By \eqref{eq:proj-vi}, 
    \[
    x^*=\proj_{K,\Dm^{-1}}(z) \quad \text{iff}\quad  \inner{x^*-z}{\Dm^{-1}(x^*-w)}\le0 \quad \text{for all }w\in K.
    \]
With $z=x^*-\alpha\Dm F(x^*)$ one has $x^*-z=\alpha\Dm F(x^*)$, so the condition becomes
\[
\alpha\inner{F(x^*)}{x^*-w}\le0, \quad \text{i.e.} \quad \inner{F(x^*)}{w-x^*}\ge0\quad \text{for all } w\in K,
\]
which means that $x^*$ solves \ref{eq:vi} and therefore, by Proposition~\ref{prop:gap}, $g_\alpha(x^*)=0$.

\item The map $\psi(y)=G(x^*,y)+\tfrac1{2\alpha}\norm{y-x^*}_{\Dm^{-1}}^2$ is convex with $\psi(x^*)=0$, so
$x^*$ minimizes $\psi$ over $K$ if and only if $x^*$ solves \ref{eq:ep}. Indeed, if
$G(x^*,y)\ge G(x^*,x^*)=0$ for all $y\in K$, then $\psi(y)\ge0=\psi(x^*)$ for all $y\in K$. Conversely, if
$x^*$ minimizes $\psi$ over $K$, fix $y\in K$ and set $y_t=x^*+t(y-x^*)\in K$ for $t\in(0,1]$. Convexity of
$G(x^*,\cdot)$ together with $G(x^*,x^*)=0$ gives
\[
0\le\psi(y_t)-\psi(x^*)=G(x^*,y_t)+\tfrac{t^2}{2\alpha}\norm{y-x^*}_{\Dm^{-1}}^2
\le t\,G(x^*,y)+\tfrac{t^2}{2\alpha}\norm{y-x^*}_{\Dm^{-1}}^2,
\]
and dividing by $t$ and letting $t\downarrow0$ yields $G(x^*,y)\ge0$. Hence $x^*$ solves \ref{eq:ep};
equivalently $g_\alpha^{\mathrm{EP}}(x^*)=0$ by Proposition~\ref{prop:epgap}.
\end{enumerate}

\end{proof}

Each inner subproblem has a strictly convex (resp.\ concave) objective, so its minimizer (resp.\
maximizer)---the algorithm's step $y_k$---is the unique point of the associated $\argmin$ (resp.\
$\argmax$). Here $d_k=y_k-x_k$ is the gap-defining direction, which lets a single algorithmic template
serve all three classes.

% ============================================================================
\section{The scaled gap-function algorithm for VI/MVI/EP}\label{sec:algorithm}
% ============================================================================

We write $g_\alpha$ for the active gap function---$g_\alpha$ itself for VI, $g_\alpha^{\mathrm{MVI}}$ for MVI,
$g_\alpha^{\mathrm{EP}}$ for EP.
Algorithm~\ref{alg:sgfm} states the scaled gap-function method (SGFM). Two \emph{certificates} $v_k$ drive the line search---a
certificate being a vector whose pairing $\inner{v_k}{d_k}$ verifies that $d_k$ decreases the merit:
\begin{itemize}[leftmargin=2em]
  \item[\textbf{(V1)}] \emph{Smooth-gap descent.} For $F\in C^1$, take $v_k=\nabla g_\alpha(x_k)$.
  By Proposition~\ref{prop:gap}, $d_k$ is a descent direction for $g_\alpha$ when the symmetric part
  of $\nabla F(x_k)$ is positive definite. The test~\eqref{eq:sgm-ls-gap} is then the base
  rule~\eqref{eq:sgm-ls} with $f:=g_\alpha$.
  \item[\textbf{(V2)}] \emph{Residual / separating-hyperplane.} For merely pseudomonotone or
  nonsmooth data, replace ``decrease the merit'' by a Solodov--Svaiter separating-hyperplane test on
  the scaled natural residual
  \[
    r(x)=\norm{x-\proj_{K,\Dm^{-1}}(x-\alpha\Dm F(x))}_{\Dm^{-1}},
  \]
  measured against the moving average $\mathcal T_k$~\cite{Solodov1999}. Only (V1) is analysed in this
  paper. For (V2), see Remark~\ref{rem:flagship}.
\end{itemize}

\begin{algorithm}[H]
\caption{Scaled Gap-Function Method (SGFM) for VI/MVI/EP}\label{alg:sgfm}
\KwInitialization{$x_0\in K$; $0<\delta_1<\sigma<1$, $\delta_2>0$, $0<\lambda_{\max}\le1$, $\lambda\in(0,1]$ (MVI fixed step), $\mu\ge1$; ranges
$0<\alpha_{\min}\le\alpha_k\le\alpha_{\max}$, $0\le\eta_k\le\eta_{\max}<1$;
set $\mathcal T_0=g_\alpha(x_0)$.}
\For{$k=0,1,2,\dots$}{
  choose a scaling matrix $\Dm_k\in\Dfam_\mu$\;
  compute the scaled (generalized) projection $y_k$ and $d_k=y_k-x_k$:\;
  \quad VI: $y_k=\proj_{K,\Dm_k^{-1}}(x_k-\alpha_k\Dm_k F(x_k))$\;
  \quad MVI: $y_k\in\argmin_{y\in K}\{\inner{F(x_k)}{y-x_k}+\varphi(y)+\tfrac{1}{2\alpha_k}\norm{y-x_k}_{\Dm_k^{-1}}^2\}$\;
  \quad EP: $y_k\in\argmin_{y\in K}\{G(x_k,y)+\tfrac{1}{2\alpha_k}\norm{y-x_k}_{\Dm_k^{-1}}^2\}$\;
  \If{$d_k=0$}{\textbf{stop}: $x_k$ solves the problem\;}
  \eIf{$g_\alpha$ is smooth \textnormal{(VI, EP)}}{
    form the certificate $v_k=\nabla g_\alpha(x_k)$ and set $s_k=-\inner{v_k}{d_k}/\norm{d_k}^2$\;
    find the smallest integer $j_k\ge0$ such that $\lambda_k=\min\{s_k,\lambda_{\max}\}\,\sigma^{j_k}$ satisfies\;
    \quad $g_\alpha(x_k+\lambda_k d_k)\le \mathcal T_k+\delta_1\lambda_k\inner{v_k}{d_k}-\delta_2\lambda_k^2\norm{d_k}^2$\;
    set $x_{k+1}=x_k+\lambda_k d_k$ and $\mathcal T_{k+1}=\eta_{k+1}\mathcal T_k+(1-\eta_{k+1})\,g_\alpha(x_{k+1})$\;
  }{
    take a fixed relaxation $\lambda\in(0,1]$ (nonsmooth MVI gap) and set $x_{k+1}=x_k+\lambda d_k$\;
  }
}
\end{algorithm}

In an implementation the exact test $d_k=0$ is replaced by $\norm{d_k}\le\varepsilon$ for a tolerance
$\varepsilon>0$, with an iteration cap as a safeguard (Section~\ref{ss:setup-stop}).

The line-search test in Algorithm~\ref{alg:sgfm} reads
\begin{equation}\label{eq:sgm-ls-gap}
  g_\alpha(x_k+\lambda_k d_k)\le \mathcal T_k+\delta_1\lambda_k\inner{v_k}{d_k}-\delta_2\lambda_k^2\norm{d_k}^2 .
\end{equation}

\begin{as}\label{as:standing}
\begin{enumerate}[label=(\roman*),leftmargin=2em]
  \item $K\subseteq\R^n$ is nonempty, closed, and convex, and $\Sol\neq\emptyset$.
  \item The relevant gap function ($g_\alpha$ for VI/MVI, $g_\alpha^{\mathrm{EP}}$ for EP) is coercive, so
  its sublevel sets are bounded. Under strong monotonicity this is automatic, since the error bound makes
  the gap coercive.
  \item $F$ is continuous (resp.\ $G$ is continuous).
  \item The scaling matrices satisfy $\Dm_k\in\Dfam_\mu$ for all $k$.
\end{enumerate}
\end{as}

\begin{as}[Direction/sufficient-decrease condition]\label{as:dir}
There is $\kappa>0$ such that the certificate $v_k$ and direction $d_k$ satisfy
$\inner{v_k}{d_k}\le -\kappa\,\norm{d_k}^2$ for all $k$.
\end{as}

% ============================================================================
\section{Convergence analysis}\label{sec:convergence}
% ============================================================================

Throughout this section we analyse the smooth-gap certificate (V1) for the variational inequality
\ref{eq:vi} and fix the scaling matrix and regularization parameter: $\Dm_k\equiv\Dm\in\Dfam_\mu$ and
$\alpha>0$ constant, so that the gap function $g_\alpha$ of \eqref{eq:gap} is a fixed $C^1$ function. We add
the following standing hypotheses to Assumption~\ref{as:standing}.

\begin{as}[V1 setting]\label{as:v1}
$F\in C^{1,1}$ is strongly monotone with modulus $m>0$ and $L$-Lipschitz. We take
$v_k=\nabla g_\alpha(x_k)$.
\end{as}

The following example exhibits an operator satisfying Assumptions~\ref{as:standing}--\ref{as:v1}.

\begin{example}[A genuine VI satisfying Assumptions~\ref{as:standing}--\ref{as:v1}]\label{ex:vi}
Let $n=2$, $K=[-1,1]^2$, and $F(x)=Mx+q$ with
\[
  M=mI+J,\qquad m=1,\qquad J=\begin{pmatrix}0&2\\-2&0\end{pmatrix},\qquad q=\begin{pmatrix}-1\\\tfrac12\end{pmatrix},
\]
so that $J^\top=-J$, and take the scaling $\Dm=I\in\Dfam_\mu$ with $\mu=1$.
\begin{itemize}[leftmargin=2em]
  \item \emph{Strong monotonicity, modulus $m=1$:} $\inner{F(x)-F(y)}{x-y}=\inner{M(x-y)}{x-y}=m\norm{x-y}^2$,
  since $\inner{Jd}{d}=0$ for skew-symmetric $J$.
  \item \emph{Lipschitz, $L=\norm{M}_2=\sqrt5$,} because $M^\top M=5I$.
  \item \emph{A genuine VI:} $F$ is affine with constant Jacobian $\nabla F=M$, whose symmetric part
  $\tfrac12(M+M^\top)=mI\succ0$ makes the scaled-projection step $d=y_\alpha(x)-x$ a descent direction for
  $g_\alpha$ (Proposition~\ref{prop:gap}). Here $y_\alpha(x)=\proj_{K,\Dm^{-1}}\!\big(x-\alpha\Dm F(x)\big)$,
  which for $\Dm=I$ is the componentwise clipping of $x-\alpha F(x)$ onto $[-1,1]^2$. The skew part $J\neq0$
  makes $\nabla F$ asymmetric, so $F$ is not a gradient and the problem is not an optimization.
\end{itemize}
Since $K$ is compact convex and $F$ is continuous and strongly monotone, $\mathrm{VI}(F,K)$ has a unique
solution, and Assumptions~\ref{as:standing} and~\ref{as:v1} hold.
\end{example}

By Definition~\ref{def:mono} the symmetric part of $\nabla F$ satisfies
$\inner{d}{\nabla F(x)\,d}\ge m\norm{d}^2$, so the descent identity \eqref{eq:gap-descent} gives
\begin{equation}\label{eq:dir-m}
  \inner{\nabla g_\alpha(x_k)}{d_k}\ \le\ -\,m\,\norm{d_k}^2 ,
\end{equation}
i.e. Assumption~\ref{as:dir} holds with $\kappa=m$, and $s_k=-\inner{v_k}{d_k}/\norm{d_k}^2\ge\kappa$.
Problem~\ref{eq:vi} has a unique solution, $x^*$, because $F$ is strongly monotone. By
Proposition~\ref{prop:eb} the error bound \eqref{eq:eb} holds globally:
\[
  g_\alpha(x)\ \ge\ \tau\,\norm{x-x^*}^2 \qquad\forall\, x\in K .
\]
Here $\tau>0$ exists for every $\alpha>0$. The explicit value $\tau=(1-q)^2/(2\alpha\mu)$
(Proposition~\ref{prop:merit-equiv}) requires the additional rate restriction $\alpha<2m/(\mu^3L^2)$. This
makes $g_\alpha$ coercive, so the sublevel set $\mathcal L_0=\{x\in K: g_\alpha(x)\le \mathcal T_0\}$ is
bounded.

The analysis below uses the following set and estimates. Set
\[
  \mathcal C_0=\operatorname{conv}\big(\mathcal L_0\cup y_\alpha(\mathcal L_0)\big)\subseteq K ,
\]
where $y_\alpha$ is the scaled projection step \eqref{eq:gap-max}. It is a bounded convex set containing
the segment $[x,y_\alpha(x)]$ for every $x\in\mathcal L_0$. Because $F\in C^{1,1}$ and $\mathcal C_0$ is
bounded, $\nabla g_\alpha$ is Lipschitz on $\mathcal C_0$ with a constant $L_g$.

By Lemma~\ref{lem:descent}(ii) the iterates satisfy $x_k\in\mathcal L_0$, so $\{x_k\}$ is bounded. Since
$F$ is $L$-Lipschitz, this gives $M_F:=\sup_k\norm{F(x_k)}<\infty$. Finally, at the maximizer $y_\alpha(x)$
the gap equals
\[
  g_\alpha(x)=-\inner{F(x)}{d}-\tfrac1{2\alpha}\norm{d}_{\Dm^{-1}}^2 ,\qquad d=y_\alpha(x)-x ,
\]
so
\begin{equation}\label{eq:gap-ub}
  0\ \le\ g_\alpha(x)\ \le\ -\inner{F(x)}{d}\ \le\ \norm{F(x)}\,\norm{d} .
\end{equation}

\begin{lemma}[Line search and sufficient decrease]\label{lem:descent}
Under Assumptions~\ref{as:standing} and~\ref{as:v1}:
\begin{enumerate}[label=(\roman*),leftmargin=2em]
  \item the line search terminates and
  $\lambda_k\ge\lambda_{\min}:=\min\{\kappa,\lambda_{\max},\sigma\bar\lambda\}>0$, where $\kappa$ is the
  direction constant of Assumption~\ref{as:dir} (here $\kappa=m$ by~\eqref{eq:dir-m}) and
  \[
  \bar\lambda=(1-\delta_1)\kappa/(L_g/2+\delta_2).
  \]
  \item $g_\alpha(x_k)\le\mathcal T_k$ for all $k$, and $\{\mathcal T_k\}$ is nonincreasing.
  \item $\mathcal T_k-\mathcal T_{k+1}\ge (1-\eta_{\max})\,\rho\,\norm{d_k}^2$ with
  $\rho=\delta_1\kappa\lambda_{\min}>0$.
\end{enumerate}
\end{lemma}
\begin{proof}
We argue by induction that $g_\alpha(x_k)\le\mathcal T_k$ for all $k$. This holds at $k=0$ because
$\mathcal T_0=g_\alpha(x_0)$. Assume it at index $k$, so $x_k\in\mathcal L_0$.
\begin{enumerate}[label=(\roman*),leftmargin=2em]
  \item For $0<\lambda\le\lambda_{\max}\le1$ the trial point $x_k+\lambda d_k$ lies on
  $[x_k,y_\alpha(x_k)]\subseteq\mathcal C_0$, so the descent lemma for the $L_g$-Lipschitz gradient on
  $\mathcal C_0$ gives
  \[
    g_\alpha(x_k+\lambda d_k)\ \le\ g_\alpha(x_k)+\lambda\inner{\nabla g_\alpha(x_k)}{d_k}
    +\tfrac{L_g}{2}\lambda^2\norm{d_k}^2 .
  \]
  With $g_\alpha(x_k)\le\mathcal T_k$ and $v_k=\nabla g_\alpha(x_k)$, the test \eqref{eq:sgm-ls-gap} holds
  whenever
  \[
    \big(\tfrac{L_g}{2}+\delta_2\big)\lambda\norm{d_k}^2\ \le\ (1-\delta_1)\big(-\inner{\nabla g_\alpha(x_k)}{d_k}\big) ,
  \]
  which by \eqref{eq:dir-m} is true for all $0<\lambda\le\bar\lambda$. Hence the backtracking terminates.
  Its first trial is $\min\{s_k,\lambda_{\max}\}\ge\min\{\kappa,\lambda_{\max}\}$ and the trials shrink by
  $\sigma\in(0,1)$. A trial fails only when it exceeds $\bar\lambda$, so the accepted step satisfies
  $\lambda_k\ge\min\{\kappa,\lambda_{\max},\sigma\bar\lambda\}=\lambda_{\min}$.

  \item The accepted step obeys \eqref{eq:sgm-ls-gap}, hence
  \[
    g_\alpha(x_{k+1})\ \le\ \mathcal T_k+\delta_1\lambda_k\inner{\nabla g_\alpha(x_k)}{d_k}
    -\delta_2\lambda_k^2\norm{d_k}^2\ \le\ \mathcal T_k .
  \]
  Therefore $\mathcal T_{k+1}=\eta_{k+1}\mathcal T_k+(1-\eta_{k+1})g_\alpha(x_{k+1})\le\mathcal T_k$, and
  \[
  \mathcal T_{k+1}-g_\alpha(x_{k+1})=\eta_{k+1}\big(\mathcal T_k-g_\alpha(x_{k+1})\big)\ge0.
  \]
  This closes the
  induction and shows $\{\mathcal T_k\}$ is nonincreasing.

  \item Using \eqref{eq:dir-m}, $\lambda_k\ge\lambda_{\min}$, and $\eta_{k+1}\le\eta_{\max}$,
  \[
  \begin{array}{lcl}
    \mathcal T_k-\mathcal T_{k+1}&=&(1-\eta_{k+1})\big(\mathcal T_k-g_\alpha(x_{k+1})\big) \\
    \ &\ge& \ (1-\eta_{\max})\big(-\delta_1\lambda_k\inner{\nabla g_\alpha(x_k)}{d_k}\big) \\
    \ &\ge&\ (1-\eta_{\max})\,\delta_1\kappa\lambda_{\min}\norm{d_k}^2=(1-\eta_{\max})\,\rho\,\norm{d_k}^2 .
  \end{array}
  \]
\end{enumerate}
\end{proof}

\begin{remark}[Reusability of Lemma~\ref{lem:descent}]\label{rem:descent-generic}
The proof of Lemma~\ref{lem:descent} uses only that $g_\alpha$ is nonnegative and $C^1$ with
$L_g$-Lipschitz gradient on $\mathcal C_0$, and the direction bound $\inner{v_k}{d_k}\le-\kappa\norm{d_k}^2$
of Assumption~\ref{as:dir}. The lemma therefore holds with $g_\alpha$ replaced by any function with these
properties. Theorem~\ref{thm:ep} applies it to the equilibrium gap $g_\alpha^{\mathrm{EP}}$.
\end{remark}

\begin{theorem}[Global convergence, V1]\label{thm:global}
Under Assumptions~\ref{as:standing} and~\ref{as:v1}, the series $\sum_{k}\norm{d_k}^2$ is convergent and
the iterates $x_k$ converge to the unique solution $x^*$ of \ref{eq:vi}, with $g_\alpha(x_k)\to0$.
\end{theorem}
\begin{proof}
By Lemma~\ref{lem:descent}(iii), summed over $k=0,\dots,N$ and telescoped,
\[
  (1-\eta_{\max})\,\rho\sum_{k=0}^{N}\norm{d_k}^2\ \le\ \mathcal T_0-\mathcal T_{N+1}\ \le\ \mathcal T_0 ,
\]
because $\mathcal T_{N+1}\ge g_\alpha(x_{N+1})\ge0$. Letting $N\to\infty$ gives $\sum_{k}\norm{d_k}^2<\infty$,
hence $\norm{d_k}\to0$. By \eqref{eq:gap-ub}, 
\[
g_\alpha(x_k)\le\norm{F(x_k)}\norm{d_k}\le M_F\norm{d_k}\to0.
\]
Finally, the error bound $g_\alpha(x_k)\ge\tau\norm{x_k-x^*}^2$ gives
$\norm{x_k-x^*}^2\le g_\alpha(x_k)/\tau\to0$, so $x_k\to x^*$.
\end{proof}

\begin{corollary}[Best-iterate sublinear rate]\label{cor:sublinear}
Under Assumptions~\ref{as:standing} and~\ref{as:v1}, for every $N$,
\[
  \min_{0\le k\le N}\norm{d_k}^2\ \le\ \frac{\mathcal T_0}{(1-\eta_{\max})\,\rho\,(N+1)} ,
\]
so $\displaystyle \min_{0\le k\le N}\norm{d_k}=O(1/\sqrt{N})$ and, by \eqref{eq:gap-ub},
$\displaystyle\min_{0\le k\le N}g_\alpha(x_k)=O(1/\sqrt{N})$. The bound is informative only when the direction
constant is positive (here $\kappa=m>0$, so $\rho=\delta_1\kappa\lambda_{\min}>0$). By
Remark~\ref{rem:descent-generic} it holds for any merit satisfying the direction condition, so a merely
monotone operator, for which the direction condition holds only with $\kappa=0$, is not covered.
\end{corollary}
\begin{proof}
The proof of Theorem~\ref{thm:global} gives $\displaystyle (1-\eta_{\max})\,\rho\sum_{k=0}^{N}\norm{d_k}^2\le\mathcal T_0,$ and $\displaystyle\min_{0\le k\le N}\norm{d_k}^2\le\tfrac{1}{N+1}\sum_{k=0}^{N}\norm{d_k}^2$ 
yields the bound
\[
  \min_{0\le k\le N}\norm{d_k}^2\ \le\ \frac{\mathcal T_0}{(1-\eta_{\max})\,\rho\,(N+1)}.
\]
The gap estimate follows from $g_\alpha(x_k)\le M_F\norm{d_k}$ in \eqref{eq:gap-ub}.
\end{proof}

\begin{theorem}[R-linear convergence, V1]\label{thm:rate}
Under Assumptions~\ref{as:standing} and~\ref{as:v1}, suppose moreover that
$\alpha\in\big(0,\,2m/(\mu^3L^2)\big)$, and set
\[
  q=\sqrt{1-\tfrac{2\alpha m}{\mu}+\alpha^2\mu^2L^2}\in[0,1),\qquad \theta=1-\lambda_{\min}(1-q)\in(0,1).
\]
Then the iterates $x_k$ converge to $x^*$ R-linearly with ratio $\theta$, that is,
\begin{equation}\label{eq:rate}
  \norm{x_k-x^*}\ \le\ \mu\,\theta^{k}\,\norm{x_0-x^*}\qquad\text{for all }k,
\end{equation}
and the gap converges R-linearly with the same ratio, $g_\alpha(x_k)\le C_g\,\theta^{k}$ for some $C_g>0$.
\end{theorem}
\begin{proof}
Set $e_k=x_k-x^*$ and write $\norm{\cdot}_*=\norm{\cdot}_{\Dm^{-1}}$. The step map
$T(x)=\proj_{K,\Dm^{-1}}(x-\alpha\Dm F(x))=y_\alpha(x)$ has $x^*$ as a fixed point by
Proposition~\ref{prop:fixed}, and $y_k=T(x_k)$. By Remark~\ref{rem:scaling}, $\norm{x}_{\Dm}$ and
$\norm{x}_*$ lie in $[\tfrac{1}{\sqrt\mu}\norm{x},\,\sqrt\mu\,\norm{x}]$ for every $x$.

We first show that $T$ contracts in $\norm{\cdot}_*$. As $\proj_{K,\Dm^{-1}}$ is nonexpansive in
$\norm{\cdot}_*$,
\begin{equation}\label{eq:rate-exp}
  \norm{y_k-x^*}_*^2\le\norm{e_k-\alpha\Dm\big(F(x_k)-F(x^*)\big)}_*^2
  =\norm{e_k}_*^2-2\alpha\inner{e_k}{F(x_k)-F(x^*)}+\alpha^2\norm{F(x_k)-F(x^*)}_{\Dm}^2 .
\end{equation}
Strong monotonicity and $L$-Lipschitz continuity, together with the norm bounds above, give
\begin{equation}\label{eq:rate-bd}
  \inner{e_k}{F(x_k)-F(x^*)}\ge\tfrac{m}{\mu}\norm{e_k}_*^2 ,\qquad
  \norm{F(x_k)-F(x^*)}_{\Dm}^2\le\mu^2L^2\norm{e_k}_*^2 .
\end{equation}
Substituting \eqref{eq:rate-bd} into \eqref{eq:rate-exp} yields
\begin{equation}\label{eq:rate-q}
  \norm{y_k-x^*}_*\ \le\ q\,\norm{e_k}_* ,\qquad q^2=1-\tfrac{2\alpha m}{\mu}+\alpha^2\mu^2L^2 ,
\end{equation}
and $q<1$ because $\alpha<2m/(\mu^3L^2)$. Since $x_{k+1}=(1-\lambda_k)x_k+\lambda_k y_k$ with
$\lambda_k\in[\lambda_{\min},1]$, the relaxed step inherits the contraction \eqref{eq:rate-q}:
\begin{equation}\label{eq:rate-relax}
  \norm{e_{k+1}}_*\le(1-\lambda_k)\norm{e_k}_*+\lambda_k\norm{y_k-x^*}_*
  \le\big(1-\lambda_k(1-q)\big)\norm{e_k}_*\le\theta\,\norm{e_k}_* .
\end{equation}
Iterating \eqref{eq:rate-relax} gives $\norm{e_k}_*\le\theta^k\norm{e_0}_*$. Converting back to the
Euclidean norm,
\[
  \norm{x_k-x^*}\le\sqrt\mu\,\norm{e_k}_*\le\sqrt\mu\,\theta^k\norm{e_0}_*\le\mu\,\theta^k\norm{x_0-x^*} ,
\]
which is \eqref{eq:rate}. Finally, $\norm{d_k}\le\big(1+\mu(1+\alpha\mu L)\big)\norm{e_k}$ by Lipschitz
continuity of $y_\alpha$ together with $y_\alpha(x^*)=x^*$, so \eqref{eq:gap-ub} gives $g_\alpha(x_k)\le M_F\norm{d_k}\le C_g\,\theta^k$.
\end{proof}

\begin{remark}\label{rem:firstorder}
The R-linear rate does not come from the sufficient decrease of Lemma~\ref{lem:descent} alone. Its
per-step estimate $\mathcal T_k-\mathcal T_{k+1}\ge(1-\eta_{\max})\rho\,\norm{d_k}^2$ is quadratic in the
residual, so it controls only $\sum_k\norm{d_k}^2$ and gives the sublinear rate of
Corollary~\ref{cor:sublinear}. An R-linear rate from this estimate alone would need the per-step decrease
to be proportional to $g_\alpha(x_k)$, that is, a bound $g_\alpha(x)\le C\norm{d}^2$. No such bound holds:
by \eqref{eq:gap-lb} and \eqref{eq:gap-ub},
\[
  \tfrac{1}{2\alpha}\norm{d}_{\Dm^{-1}}^2\ \le\ g_\alpha(x)\ \le\ \norm{F(x)}\,\norm{d} ,
\]
and the linear upper bound is sharp at a solution $x^*$ on the boundary of $K$, where $d=0$ but
$\nabla g_\alpha(x^*)=F(x^*)\ne0$ (Remark~\ref{rem:measure}).

The R-linear rate of Theorem~\ref{thm:rate} comes instead from the contraction \eqref{eq:rate-q}, which
strong monotonicity provides. With the error bound \eqref{eq:eb} and the bounds \eqref{eq:gap-lb} and
\eqref{eq:gap-ub}, it makes the distance $\norm{x_k-x^*}$, the residual $\norm{d_k}$, and the gap
$g_\alpha(x_k)$ converge to $0$ R-linearly with the same ratio $\theta$ (Corollary~\ref{cor:allmeasures}).
For the same reason, convergence is measured by $g_\alpha(x_k)$ or $\norm{d_k}$ rather than by
$\norm{\nabla g_\alpha(x_k)}$, which need not tend to $0$.
\end{remark}

\begin{remark}[Rate on Example~\ref{ex:vi}]\label{rem:exrate}
The rate window of Theorem~\ref{thm:rate} is nonempty. With $\Dm=I$ ($\mu=1$), $m=1$, and $L=\sqrt5$, it is
$2m/(\mu^3L^2)=0.4$, and $\alpha=0.2$ gives
\[
  q=\sqrt{1-\tfrac{2\alpha m}{\mu}+\alpha^2\mu^2L^2}=\sqrt{0.8}\approx0.894\in[0,1) ,
\]
so the iterates converge R-linearly.
\end{remark}

The next example shows that $F$ need not be affine.

\begin{example}[A nonlinear VI satisfying Assumptions~\ref{as:standing}--\ref{as:v1}]\label{ex:viquad}
Let $n=2$, $K=[0,1]^2$, and $F(x)=\nabla\psi(x)+Jx+q$ with
\[
  \psi(x)=\tfrac12\norm{x}^2+\tfrac16\big(x_1^3+x_2^3\big),\qquad
  J=\begin{pmatrix}0&2\\-2&0\end{pmatrix},\qquad
  q=\begin{pmatrix}-\tfrac32\\\tfrac12\end{pmatrix},
\]
so that $F(x)=\big(x_1+\tfrac12x_1^2+2x_2-\tfrac32,\ -2x_1+x_2+\tfrac12x_2^2+\tfrac12\big)^\top$ is quadratic,
and take the scaling $\Dm=I\in\Dfam_\mu$ with $\mu=1$. Its Jacobian is
\[
  \nabla F(x)=\begin{pmatrix}1+x_1&2\\-2&1+x_2\end{pmatrix}.
\]
\begin{itemize}[leftmargin=2em]
  \item \emph{Strong monotonicity, modulus $m=1$:} the symmetric part of $\nabla F(x)$ is
  $\operatorname{diag}(1+x_1,1+x_2)\succeq I$ on $K$, so integrating $\nabla F$ along the segment
  $[y,x]\subseteq K$ gives $\inner{F(x)-F(y)}{x-y}\ge\norm{x-y}^2$, the skew part contributing nothing to
  the quadratic form.
  \item \emph{Lipschitz, $L=3$:} $L=\max_{x\in K}\norm{\nabla F(x)}_2=3$, attained at $x=(1,0)$, where
  $\nabla F^\top\nabla F$ has largest eigenvalue $9$.
  \item \emph{A genuine nonlinear VI:} $F$ is quadratic, and $\nabla F$ carries the nonzero constant skew
  part $J$, so $F$ is not a gradient. The problem is neither affine nor an optimization.
\end{itemize}
Since $K$ is compact convex and $F$ is continuous, strongly monotone, and $C^{1,1}$ on $K$,
$\mathrm{VI}(F,K)$ has a unique solution, and Assumptions~\ref{as:standing} and~\ref{as:v1} hold. The rate
window of Theorem~\ref{thm:rate} is nonempty, $2m/(\mu^3L^2)=\tfrac29$, and $\alpha=\tfrac19$ gives
$q=\sqrt{1-\tfrac{2\alpha m}{\mu}+\alpha^2\mu^2L^2}=\sqrt{\tfrac89}=\tfrac{2\sqrt2}{3}\approx0.943\in[0,1)$, so
the iterates converge R-linearly.
\end{example}

\subsection{Error bound and merit equivalence}\label{ss:meriteq}

Theorem~\ref{thm:rate} measures convergence through the regularized gap $g_\alpha$. We now relate this gap to the other quantities that measure progress toward the solution, chiefly the
distance and the residual. The next proposition answers this, and the same relations are reused for
equilibrium problems (Section~\ref{ss:ep}).

\begin{proposition}[Error bound and merit equivalence, V1]\label{prop:merit-equiv}
Under Assumptions~\ref{as:standing} and~\ref{as:v1}, suppose moreover that
$\alpha\in\big(0,2m/(\mu^3L^2)\big)$, and set $q=\sqrt{1-\tfrac{2\alpha m}{\mu}+\alpha^2\mu^2L^2}\in[0,1)$.
Write $\norm{\cdot}_*=\norm{\cdot}_{\Dm^{-1}}$. For $x\in K$ and $t>0$, the scaled residual at step size $t$ is
\[
  d_t(x)=y_t(x)-x ,\qquad y_t(x)=\proj_{K,\Dm^{-1}}\!\big(x-t\Dm F(x)\big).
\]
In particular, $d_\alpha(x)=y_\alpha(x)-x$ is the residual $d$ of \eqref{eq:gap-lb}, with
$d_k=d_\alpha(x_k)$. Then the following hold for all $x\in K$.
\begin{enumerate}[label=(\roman*),leftmargin=2em]
  \item The scaled residual satisfies,
  \begin{equation}\label{eq:res-eb}
    (1-q)\,\norm{x-x^*}_*\ \le\ \norm{d_\alpha(x)}_* ,\qquad
    \norm{d_t(x)}\ \le\ \big(1+\mu(1+t\mu L)\big)\,\norm{x-x^*}\quad(t>0).
  \end{equation}
  \item The regularized gap is an explicit error bound for the distance to $\Sol=\{x^*\}$,
  \begin{equation}\label{eq:eb-explicit}
    g_\alpha(x)\ \ge\ \tau\,\dist(x,\Sol)^2 ,\qquad \tau=\frac{(1-q)^2}{2\alpha\mu} .
  \end{equation}
  \item For $0<\alpha<\beta$ the D-gap satisfies,
  \begin{equation}\label{eq:dgap-twosided}
    \tfrac12\big(\tfrac1\alpha-\tfrac1\beta\big)\norm{d_\alpha(x)}_*^2
    \ \le\ g_{\alpha\beta}(x)\ \le\
    \tfrac12\big(\tfrac1\alpha-\tfrac1\beta\big)\norm{d_\beta(x)}_*^2 .
  \end{equation}
  \item If moreover $\beta<2m/(\mu^3L^2)$, then $\dist(x,\Sol)^2$, $\norm{d_\alpha(x)}^2$,
  $\norm{d_\beta(x)}^2$, and $g_{\alpha\beta}(x)$ are mutually equivalent on $K$, i.e.\ there are constants
  $0<\underline{c}\le\overline{c}$ such that
  \[
    \underline{c}\,\dist(x,\Sol)^2\ \le\ \norm{d_\alpha(x)}^2,\ \ \norm{d_\beta(x)}^2,\ \ g_{\alpha\beta}(x)
    \ \le\ \overline{c}\,\dist(x,\Sol)^2 \qquad(x\in K).
  \]
\end{enumerate}
\end{proposition}
\begin{proof}
\begin{enumerate}[label=(\roman*),leftmargin=2em]
  \item The proof of Theorem~\ref{thm:rate} shows that 
  \[
  \norm{y_\alpha(x)-x^*}_*\le q\norm{x-x^*}_*, \quad \text{where}\quad y_\alpha(x^*)=x^*.
  \]
  Therefore
  \[
    \norm{x-x^*}_*\ \le\ \norm{d_\alpha(x)}_*+\norm{y_\alpha(x)-x^*}_*
    \ \le\ \norm{d_\alpha(x)}_*+q\,\norm{x-x^*}_* ,
  \]
  which rearranges to the first inequality of \eqref{eq:res-eb}. For the second, fix $t>0$. Since $x^*$
  solves \ref{eq:vi}, it is a fixed point of the scaled projection step at every step size, so
  $y_t(x^*)=x^*$ (Proposition~\ref{prop:fixed}). The projection $\proj_{K,\Dm^{-1}}$ is Lipschitz with
  constant $\mu$ in the Euclidean norm (Remark~\ref{rem:proj}, since $\Dm^{-1}\in\Dfam_\mu$). Writing
  $u=x-t\Dm F(x)$ and $u^*=x^*-t\Dm F(x^*)$, so that $y_t(x)=\proj_{K,\Dm^{-1}}(u)$ and
  $y_t(x^*)=\proj_{K,\Dm^{-1}}(u^*)=x^*$,
  \[
    \norm{y_t(x)-x^*}\ \le\ \mu\,\norm{u-u^*}
    \ =\ \mu\,\norm{(x-x^*)-t\Dm\big(F(x)-F(x^*)\big)} .
  \]
  By the triangle inequality, $\norm{\Dm}\le\mu$ (Remark~\ref{rem:scaling}), and the $L$-Lipschitz
  continuity of $F$,
  \[
    \norm{u-u^*}\ \le\ \norm{x-x^*}+t\,\norm{\Dm}\,\norm{F(x)-F(x^*)}\ \le\ (1+t\mu L)\,\norm{x-x^*} ,
  \]
  so that $\norm{y_t(x)-x^*}\le\mu(1+t\mu L)\norm{x-x^*}$. Finally
  $d_t(x)=\big(y_t(x)-x^*\big)-\big(x-x^*\big)$, and the triangle inequality gives
  \[
    \norm{d_t(x)}\ \le\ \norm{y_t(x)-x^*}+\norm{x-x^*}\ \le\ \big(1+\mu(1+t\mu L)\big)\norm{x-x^*} ,
  \]
  which is the second inequality of \eqref{eq:res-eb}.

  \item By the lower bound \eqref{eq:gap-lb}, the first inequality of \eqref{eq:res-eb}, and
  $\norm{\cdot}_*^2\ge\tfrac1\mu\norm{\cdot}^2$,
  \[
    g_\alpha(x)\ \ge\ \tfrac1{2\alpha}\norm{d_\alpha(x)}_*^2
    \ \ge\ \tfrac{(1-q)^2}{2\alpha}\norm{x-x^*}_*^2
    \ \ge\ \tfrac{(1-q)^2}{2\alpha\mu}\norm{x-x^*}^2 ,
  \]
  which is \eqref{eq:eb-explicit}.

  \item Write $\Phi_t(y)=\inner{F(x)}{x-y}-\tfrac1{2t}\norm{y-x}_{\Dm^{-1}}^2$ for the $t$-objective of
  \eqref{eq:gap}, so that $g_t(x)=\max_{y\in K}\Phi_t(y)=\Phi_t(y_t(x))$. Since $y_t(x)$ is the maximizer and
  $\norm{y_t(x)-x}_{\Dm^{-1}}=\norm{d_t(x)}_*$,
  \[
    g_\alpha(x)=\inner{F(x)}{x-y_\alpha(x)}-\tfrac1{2\alpha}\norm{d_\alpha(x)}_*^2 ,\qquad
    g_\beta(x)=\inner{F(x)}{x-y_\beta(x)}-\tfrac1{2\beta}\norm{d_\beta(x)}_*^2 .
  \]
  For the lower bound, $y_\alpha(x)\in K$ is feasible for the $\beta$-maximization, so
  \[
    g_\beta(x)\ \ge\ \Phi_\beta(y_\alpha(x))\ =\ \inner{F(x)}{x-y_\alpha(x)}-\tfrac1{2\beta}\norm{d_\alpha(x)}_*^2 .
  \]
  Subtracting the identity for $g_\alpha(x)$, the inner-product terms cancel and
  \[
    g_{\alpha\beta}(x)=g_\beta(x)-g_\alpha(x)\ \ge\ \Big(\tfrac1{2\alpha}-\tfrac1{2\beta}\Big)\norm{d_\alpha(x)}_*^2
    =\tfrac12\Big(\tfrac1\alpha-\tfrac1\beta\Big)\norm{d_\alpha(x)}_*^2 ,
  \]
  the lower bound of \eqref{eq:dgap-twosided}. Symmetrically, $y_\beta(x)\in K$ is feasible for the
  $\alpha$-maximization, so
  $g_\alpha(x)\ge\Phi_\alpha(y_\beta(x))=\inner{F(x)}{x-y_\beta(x)}-\tfrac1{2\alpha}\norm{d_\beta(x)}_*^2$.
  Subtracting this from the identity for $g_\beta(x)$ gives
  \[
    g_{\alpha\beta}(x)\ \le\ \Big(\tfrac1{2\alpha}-\tfrac1{2\beta}\Big)\norm{d_\beta(x)}_*^2
    =\tfrac12\Big(\tfrac1\alpha-\tfrac1\beta\Big)\norm{d_\beta(x)}_*^2 ,
  \]
  the upper bound.

  \item Recall $\dist(x,\Sol)=\norm{x-x^*}$. By the contraction estimate in the proof of
  Theorem~\ref{thm:rate}, applied at step size $t$,
  \[
    \norm{y_t(x)-x^*}_*\ \le\ q_t\,\norm{x-x^*}_* ,\qquad q_t=\sqrt{1-\tfrac{2tm}{\mu}+t^2\mu^2L^2} ,
  \]
  and $q_t<1$ precisely when $t<2m/(\mu^3L^2)$, which holds for both $t=\alpha$ and $t=\beta$. As in (i), the
  triangle inequality gives the residual lower bound $(1-q_t)\norm{x-x^*}_*\le\norm{d_t(x)}_*$. Combining it
  with the upper bound of (i) and the norm equivalence (Remark~\ref{rem:scaling}) gives, for
  $t\in\{\alpha,\beta\}$,
  \begin{equation}\label{eq:res-bounds}
    \frac{(1-q_t)^2}{\mu^2}\,\dist(x,\Sol)^2\ \le\ \norm{d_t(x)}^2
    \ \le\ \big(1+\mu(1+t\mu L)\big)^2\,\dist(x,\Sol)^2 ,
  \end{equation}
  so $\norm{d_\alpha(x)}^2$ and $\norm{d_\beta(x)}^2$ are each equivalent to $\dist(x,\Sol)^2$. Combining
  \eqref{eq:dgap-twosided} with the residual bounds yields
  \begin{equation}\label{eq:dgap-bounds}
    \frac{(1-q_\alpha)^2}{2\mu}\Big(\tfrac1\alpha-\tfrac1\beta\Big)\dist(x,\Sol)^2
    \ \le\ g_{\alpha\beta}(x)\ \le\
    \frac{\mu}{2}\Big(\tfrac1\alpha-\tfrac1\beta\Big)\big(1+\mu(1+\beta\mu L)\big)^2\dist(x,\Sol)^2 ,
  \end{equation}
  so $g_{\alpha\beta}(x)$ is equivalent to $\dist(x,\Sol)^2$ too. Taking $\underline{c}$ to be the least and
  $\overline{c}$ the greatest of the positive constants in \eqref{eq:res-bounds} and \eqref{eq:dgap-bounds} gives
  \[
  \underline{c}\,\dist(x,\Sol)^2\le\norm{d_\alpha(x)}^2,\norm{d_\beta(x)}^2,g_{\alpha\beta}(x)\le\overline{c}\,\dist(x,\Sol)^2.
  \]
\end{enumerate}
\end{proof}

\begin{corollary}[R-linear in every natural measure]\label{cor:allmeasures}
Under the hypotheses of Theorem~\ref{thm:rate}, the distance $\dist(x_k,\Sol)$, the residual
$\norm{d_\alpha(x_k)}$, and the regularized gap $g_\alpha(x_k)$ converge to $0$ R-linearly with ratio
$\theta$, while the squared residual and the D-gap $g_{\alpha\beta}(x_k)$ (for $\alpha<\beta<2m/(\mu^3L^2)$)
converge with ratio $\theta^2$.
\end{corollary}
\begin{proof}
Theorem~\ref{thm:rate} gives $\dist(x_k,\Sol)=\norm{x_k-x^*}\le\mu\theta^k\norm{x_0-x^*}=O(\theta^k)$, so the
distance converges R-linearly with ratio $\theta$. The residual and the gap inherit this rate: the upper
bound in \eqref{eq:res-eb} gives $\norm{d_\alpha(x_k)}\le\big(1+\mu(1+\alpha\mu L)\big)\norm{x_k-x^*}=
O(\theta^k)$, and \eqref{eq:gap-ub} gives $g_\alpha(x_k)\le M_F\norm{d_\alpha(x_k)}=O(\theta^k)$.

Squaring, $\norm{d_\alpha(x_k)}^2=O(\theta^{2k})$. For the D-gap, the upper bound in
\eqref{eq:dgap-twosided}, the norm equivalence $\norm{\cdot}_*^2\le\mu\norm{\cdot}^2$, and the residual
bound at $t=\beta$ give
$g_{\alpha\beta}(x_k)\le\tfrac{\mu}{2}\big(\tfrac1\alpha-\tfrac1\beta\big)\norm{d_\beta(x_k)}^2=O(\theta^{2k})$.
Hence the squared residual and the D-gap converge with ratio $\theta^2$.
\end{proof}

\subsection{Variable metric}\label{ss:varmetric}

Theorems~\ref{thm:global}--\ref{thm:rate} fix the scaling matrix. We now let it vary with $k$, taking
$\Dm_k\in\Dfam_\mu$, and write $g_{\alpha,\Dm_k}$ for the gap \eqref{eq:gap} with $\Dm=\Dm_k$. The
contraction behind Theorem~\ref{thm:rate} uses only the spectral bounds of $\Dfam_\mu$, so it holds with the
same $q$ in each $\Dm_k^{-1}$-norm. The moving-average reference of the line search compares values of the
changing merit $g_{\alpha,\Dm_k}$, so here we drive the iteration with a fixed relaxation step
$\lambda\in(0,1]$ instead. The next theorem shows that a summable change in the scaling matrices then
preserves the R-linear rate, inflating only its constant.

\begin{theorem}[Variable metric, V1]\label{thm:varmetric}
Under Assumptions~\ref{as:standing} and~\ref{as:v1}, suppose moreover that $\alpha\in(0,2m/(\mu^3L^2))$ and
$\sum_{k\ge0}\norm{\Dm_{k+1}^{-1}-\Dm_k^{-1}}<\infty$ in the spectral norm. Run the relaxed iteration
$x_{k+1}=(1-\lambda)x_k+\lambda\,y_{\alpha,\Dm_k}(x_k)$ with a fixed step $\lambda\in(0,1]$, and set
\[
  q=\sqrt{1-\tfrac{2\alpha m}{\mu}+\alpha^2\mu^2L^2}\in[0,1),\qquad \theta=1-\lambda(1-q)\in(0,1),
\]
and $M=\exp\!\big(\mu\sum_{k\ge0}\norm{\Dm_{k+1}^{-1}-\Dm_k^{-1}}\big)$. Then $x_k$ converges to $x^*$
R-linearly with ratio $\theta$, that is,
\[
  \norm{x_k-x^*}\ \le\ \mu\sqrt M\,\theta^{k}\,\norm{x_0-x^*}\qquad\text{for all }k.
\]
When $\Dm_k\equiv\Dm$, $M=1$ and the bound reduces to Theorem~\ref{thm:rate}.
\end{theorem}
\begin{proof}
Let $V_k=\norm{x_k-x^*}_{\Dm_k^{-1}}^2$. We bound $V_{k+1}$ in two steps: the step map contracts in the
current $\Dm_k^{-1}$-norm, and changing the scaling matrix perturbs $V$ by a summable factor.

In the $\Dm_k^{-1}$-norm, the contraction in the proof of Theorem~\ref{thm:rate} applies with
$\Dm=\Dm_k$. Its constant $q$ uses only the spectral bounds $[1/\mu,\mu]$, not the particular $\Dm_k$, so
\[
  \norm{y_k-x^*}_{\Dm_k^{-1}}\ \le\ q\,\norm{x_k-x^*}_{\Dm_k^{-1}}
\]
holds for every $k$. With the fixed step $x_{k+1}=(1-\lambda)x_k+\lambda y_k$, $\lambda\in(0,1]$, convexity
of $\norm{\cdot}_{\Dm_k^{-1}}$ gives
\begin{equation}\label{eq:vm-contract}
  \norm{x_{k+1}-x^*}_{\Dm_k^{-1}}\ \le\ \big(1-\lambda(1-q)\big)\norm{x_k-x^*}_{\Dm_k^{-1}}
  \ =\ \theta\,\norm{x_k-x^*}_{\Dm_k^{-1}} .
\end{equation}
Writing $\epsilon_k=\norm{\Dm_{k+1}^{-1}-\Dm_k^{-1}}$ and using $\norm{v}^2\le\mu\norm{v}_{\Dm_k^{-1}}^2$
(Remark~\ref{rem:scaling}) together with \eqref{eq:vm-contract},
\[
  V_{k+1}=\norm{x_{k+1}-x^*}_{\Dm_{k+1}^{-1}}^2
  \ \le\ \norm{x_{k+1}-x^*}_{\Dm_k^{-1}}^2+\epsilon_k\norm{x_{k+1}-x^*}^2
  \ \le\ (1+\mu\epsilon_k)\,\theta^2 V_k .
\]
Iterating this recursion from $V_0$ gives
\[
  V_k\ \le\ \Big(\prod_{j=0}^{k-1}(1+\mu\epsilon_j)\Big)\theta^{2k}V_0 .
\]
Since $\sum_j\epsilon_j<\infty$ and $1+t\le e^t$, the product is bounded uniformly in $k$,
\[
  \prod_{j=0}^{k-1}(1+\mu\epsilon_j)\ \le\ \prod_{j\ge0}(1+\mu\epsilon_j)
  \ \le\ \exp\!\Big(\mu\sum_{j\ge0}\epsilon_j\Big)\ =\ M ,
\]
so $V_k\le M\theta^{2k}V_0$. Converting back to the Euclidean norm by Remark~\ref{rem:scaling}, 
we have
\[
\begin{array}{lcl}
  \norm{x_k-x^*}^2\ \le\ \mu V_k\ &\le&\ \mu M\theta^{2k}V_0\\[2mm]
  &=& \mu M\theta^{2k}\norm{x_0-x^*}_{\Dm_0^{-1}}^2 \\[2mm]
  &\le& \mu^2 M\theta^{2k}\norm{x_0-x^*}^2 \\[2mm]
\end{array}
\]
and taking square roots gives $\norm{x_k-x^*}\le\mu\sqrt M\,\theta^k\norm{x_0-x^*}$.
\end{proof}

\begin{remark}\label{rem:varmetric}
The hypothesis $\sum_k\norm{\Dm_{k+1}^{-1}-\Dm_k^{-1}}<\infty$ makes the consecutive changes summable, so
the inverse scaling matrices converge, $\Dm_k^{-1}\to\bar\Dm^{-1}$ for some $\bar\Dm\in\Dfam_\mu$. It implies the
variable-metric condition 
\[
(1+\eta_k)\Dm_k^{-1}\succeq\Dm_{k+1}^{-1}, \quad \sum_k\eta_k<\infty,
\]
of Combettes--V\~u~\cite{Combettes2014}, where $A\succeq B$ means $A-B$ is positive semidefinite and
\[
\eta_k=\mu\norm{\Dm_{k+1}^{-1}-\Dm_k^{-1}}.
\]
Under it the variation of the scaling matrix costs nothing in the rate: the bound of
Theorem~\ref{thm:varmetric} keeps the fixed-metric ratio $\theta$ and inflates only the constant, by
$\sqrt M=\exp\!\big(\tfrac\mu2\sum_k\norm{\Dm_{k+1}^{-1}-\Dm_k^{-1}}\big)$.
\end{remark}

\subsection{Mixed variational inequalities}\label{ss:mvi}

The smooth-gap route (V1) used $g_\alpha\in C^1$, which fails when the nonsmooth term $\varphi$ is
present. The strong-monotone rate, however, rests on the contraction of the step map, and this carries
over once the projection is replaced by the scaled resolvent. For \ref{eq:mvi} the inner
subproblem is solved by the scaled forward--backward step
\[
  y_\alpha(x)=\prox^{\Dm^{-1}}_{\alpha(\varphi+\iota_K)}\!\big(x-\alpha\Dm F(x)\big)
  \in\argmin_{y\in K}\Big\{\inner{F(x)}{y-x}+\varphi(y)+\tfrac1{2\alpha}\norm{y-x}_{\Dm^{-1}}^2\Big\},
\]
with $\iota_K$ the indicator of $K$ and $\prox^{\Dm^{-1}}$ the scaled proximal operator
(Definition~\ref{def:prox}). The MVI gap is
\[
  g_\alpha^{\mathrm{MVI}}(x)=-\min_{y\in K}\Big\{\inner{F(x)}{y-x}+\varphi(y)-\varphi(x)+\tfrac1{2\alpha}\norm{y-x}_{\Dm^{-1}}^2\Big\}\ \ge\ 0,
\]
vanishing iff $x$ solves \ref{eq:mvi} (Proposition~\ref{prop:fixed}(ii) with
$G(x,y)=\inner{F(x)}{y-x}+\varphi(y)-\varphi(x)$, convex in $y$). The optimality of $y_\alpha(x)$ for
the subproblem gives (cf.\ the VI bound~\eqref{eq:gap-lb}) $g_\alpha^{\mathrm{MVI}}(x)\ge\tfrac1{2\alpha}\norm{d}_{\Dm^{-1}}^2$
with $d=y_\alpha(x)-x$.

\begin{theorem}[Mixed VI, strongly monotone]\label{thm:mvi}
Under Assumption~\ref{as:standing}, suppose moreover that $F$ is strongly monotone with modulus $m$ and
$L$-Lipschitz, $\varphi$ is proper, convex, and lower semicontinuous, and $\alpha\in(0,2m/(\mu^3L^2))$. Let
the scaling matrices be fixed ($\Dm_k\equiv\Dm$) or vary with $\sum_k\norm{\Dm_{k+1}^{-1}-\Dm_k^{-1}}<\infty$,
take a fixed relaxation step $\lambda\in(0,1]$, and set
\[
  q=\sqrt{1-\tfrac{2\alpha m}{\mu}+\alpha^2\mu^2L^2}\in[0,1),\qquad \theta=1-\lambda(1-q)\in(0,1).
\]
Then \ref{eq:mvi} has a unique solution $x^*$, and the relaxed iteration
$x_{k+1}=(1-\lambda)x_k+\lambda\,y_{\alpha,\Dm_k}(x_k)$ has the following properties.
\begin{enumerate}[label=(\roman*),leftmargin=2em]
  \item The iterates converge to $x^*$ R-linearly with ratio $\theta$,
  \[
    \norm{x_k-x^*}\ \le\ C\,\theta^{k}\,\norm{x_0-x^*}\qquad\text{for all }k,
  \]
  where $C=\mu$ when $\Dm_k\equiv\Dm$ and $C=\mu\sqrt M$ under the controlled-change condition, with $M$ as
  in Theorem~\ref{thm:varmetric}. In particular $\norm{d_k}$ and $\dist(x_k,\Sol)$ converge to $0$
  R-linearly with ratio $\theta$.
  \item If moreover $\partial\varphi$ is bounded on the iterate set, for instance $\varphi$ is Lipschitz
  near $\Sol$, then $g_\alpha^{\mathrm{MVI}}(x_k)\to0$ R-linearly with ratio $\theta$.
\end{enumerate}
\end{theorem}
\begin{proof}
Write $\norm{\cdot}_*=\norm{\cdot}_{\Dm^{-1}}$ and let
$T(x)=y_{\alpha,\Dm}(x)=\prox^{\Dm^{-1}}_{\alpha(\varphi+\iota_K)}(x-\alpha\Dm F(x))$ be the step map.

The scaled proximal map $\prox^{\Dm^{-1}}_{\alpha(\varphi+\iota_K)}$ is the resolvent of
$\alpha\Dm\,\partial(\varphi+\iota_K)$, hence nonexpansive in $\norm{\cdot}_*$~\cite{Combettes2014}. By strong
monotonicity and $L$-Lipschitz continuity of $F$ together with the spectral bounds (the estimates
\eqref{eq:rate-bd}), the forward map $I-\alpha\Dm F$ satisfies
\[
  \norm{(x-\alpha\Dm F(x))-(x'-\alpha\Dm F(x'))}_*\ \le\ q\,\norm{x-x'}_*\qquad\text{for all }x,x'\in K .
\]
Composing the two maps gives $\norm{T(x)-T(x')}_*\le q\norm{x-x'}_*$, and $q<1$ because
$\alpha<2m/(\mu^3L^2)$.

As $K$ is closed it is complete in $\norm{\cdot}_*$, so by Banach's theorem $T$ has a unique fixed point,
which by Proposition~\ref{prop:fixed}(ii) is the unique solution $x^*$ of \ref{eq:mvi}. Taking $x'=x^*$ in
the contraction gives $\norm{y_{\alpha,\Dm_k}(x_k)-x^*}_*\le q\norm{x_k-x^*}_*$, and the fixed relaxation
step inherits it,
\[
  \norm{x_{k+1}-x^*}_*\ \le\ \big(1-\lambda(1-q)\big)\norm{x_k-x^*}_*\ =\ \theta\,\norm{x_k-x^*}_* ,
\]
as in \eqref{eq:rate-relax}. Iterating and converting to the Euclidean norm gives part~(i) with $C=\mu$.
Under the controlled-change condition the argument of Theorem~\ref{thm:varmetric} replaces $\mu$ by
$\mu\sqrt M$.

For the residual, write $d(x)=T(x)-x$ and $d_k=d(x_k)$. Using $T(x^*)=x^*$, the triangle inequality and the
contraction give
\[
  \norm{x-x^*}_*\ \le\ \norm{d(x)}_*+\norm{T(x)-x^*}_*\ \le\ \norm{d(x)}_*+q\,\norm{x-x^*}_* ,
\]
so $\norm{d(x)}_*\ge(1-q)\norm{x-x^*}_*$ for all $x\in K$ (as in Proposition~\ref{prop:merit-equiv}). At the
iterates this makes $\norm{d_k}$ and $\dist(x_k,\Sol)$ R-linear, and with the subproblem lower bound
$g_\alpha^{\mathrm{MVI}}(x)\ge\tfrac1{2\alpha}\norm{d(x)}_*^2$ it yields the error bound
\[
  g_\alpha^{\mathrm{MVI}}(x)\ \ge\ \tfrac1{2\alpha}\norm{d(x)}_*^2\ \ge\ \tfrac{(1-q)^2}{2\alpha\mu}\,\dist(x,\Sol)^2\qquad\forall\,x\in K .
\]

For part~(ii), at the minimizer $y_\alpha(x)$ drop the nonnegative regularizer and apply the subgradient
inequality for $\varphi$ at $x$ (any $\xi\in\partial\varphi(x)$),
\[
  g_\alpha^{\mathrm{MVI}}(x)\ \le\ \inner{F(x)}{x-y_\alpha(x)}+\varphi(x)-\varphi(y_\alpha(x))
  \ \le\ \inner{F(x)+\xi}{x-y_\alpha(x)}\ \le\ \norm{F(x)+\xi}\,\norm{d(x)} .
\]
If $\partial\varphi$ is bounded along the iterates, the right side at $x=x_k$ is $O(\theta^k)$, so
$g_\alpha^{\mathrm{MVI}}(x_k)\to0$ R-linearly.
\end{proof}

\begin{remark}\label{rem:mvi}
When $\varphi$ is nonsmooth, $g_\alpha^{\mathrm{MVI}}$ is not $C^1$, so the smooth-gap certificate (V1)
does not apply. The rate instead comes from the resolvent contraction with a fixed relaxation step
$\lambda\in(0,1]$, or a sufficient-decrease test on the computable gap value $g_\alpha^{\mathrm{MVI}}$. A smooth merit enabling a gradient-based line search is available
through the forward--backward envelope, which we do not pursue here.
\end{remark}

We illustrate Theorem~\ref{thm:mvi} in the next example with a nonsmooth $\ell_1$ penalty.

\begin{example}[A mixed VI satisfying Theorem~\ref{thm:mvi}]\label{ex:mvi}
Keep $F(x)=Mx+q$, $M=mI+J$, $m=1$, on $K=[-1,1]^2$ from Example~\ref{ex:vi}, and add
$\varphi(x)=t\norm{x}_1$ with $t=\tfrac12$. Then $\varphi$ is proper, convex, and continuous (in particular lower semicontinuous), so
with the strong monotonicity ($m=1$) and Lipschitz continuity ($L=\sqrt5$) of $F$ the hypotheses of
Theorem~\ref{thm:mvi} hold. Its subdifferential is uniformly bounded,
\[
  \xi\in\partial\varphi(x)\ \Longrightarrow\ \norm{\xi}\le t\sqrt n=\tfrac{\sqrt2}{2},
\]
since $\partial\norm{\cdot}_1(x)\subseteq[-1,1]^n$, so the bounded-subgradient clause of
Theorem~\ref{thm:mvi}(ii) holds and $g_\alpha^{\mathrm{MVI}}(x_k)\to0$ R-linearly. With $\Dm=I$ ($\mu=1$) the rate window and ratio match Remark~\ref{rem:exrate} ($q=\sqrt{0.8}$ at
$\alpha=0.2$). For the fixed step $\lambda\in(0,1]$ the contraction factor is $\theta=1-\lambda(1-q)\in(0,1)$. The $\ell_1$ term is active at the solution. The unique solution is the interior point $x^*=(\tfrac12,0)$, at
which the optimality condition $-F(x^*)\in\partial\varphi(x^*)$ holds,
\[
  \partial\varphi(x^*)=\{\tfrac12\}\times[-\tfrac12,\tfrac12],\qquad
  -F(x^*)=\big(\tfrac12,\tfrac12\big)\in\partial\varphi(x^*).
\]
The penalty sparsifies the second coordinate, driving it from the value $0.3$ at the $\varphi\equiv0$
solution $(0.4,0.3)$ (Example~\ref{ex:vi}) to zero, so the problem is genuinely mixed.
\end{example}

\subsection{Equilibrium problems}\label{ss:ep}

Throughout this subsection $G\colon K\times K\to\R$ satisfies $G(x,x)=0$ with $G(x,\cdot)$ convex for
$x\in K$, and $G$ is differentiable with $\nabla_1 G,\nabla_2 G$ continuous on $K\times K$. By
Proposition~\ref{prop:epgap}, $g_\alpha^{\mathrm{EP}}$ is then a continuously differentiable gap function
for \ref{eq:ep}.

The inner minimization defining $g_\alpha^{\mathrm{EP}}$ has the unique minimizer
\[
  y_\alpha(x)\in\argmin_{y\in K}\Big\{G(x,y)+\tfrac1{2\alpha}\norm{y-x}_{\Dm^{-1}}^2\Big\},\qquad d=y_\alpha(x)-x.
\]
By convexity of $G(x,\cdot)$,
\begin{equation}\label{eq:ep-ub}
  g_\alpha^{\mathrm{EP}}(x)\le-G(x,y_\alpha(x))\le\inner{-\nabla_2 G(x,x)}{d}\le\norm{\nabla_2 G(x,x)}\,\norm{d},
\end{equation}
the EP analogue of \eqref{eq:gap-ub}.

\begin{theorem}[Equilibrium problems: global convergence and error bound]\label{thm:ep}
Suppose Assumption~\ref{as:standing} holds with the EP gap $g_\alpha^{\mathrm{EP}}$ as merit, that
$\nabla g_\alpha^{\mathrm{EP}}$ is Lipschitz on a bounded convex set $\mathcal C_0$ containing every segment
$[x_k,y_\alpha(x_k)]$ (automatic when $\nabla_1 G,\nabla_2 G$ are locally Lipschitz, the iterates being
bounded), and that the direction condition (Assumption~\ref{as:dir}) holds with constant
$\kappa$. Then the SGFM iteration with certificate $v_k=\nabla g_\alpha^{\mathrm{EP}}(x_k)$ has the following
properties.
\begin{enumerate}[label=(\roman*),leftmargin=2em]
  \item $\sum_k\norm{d_k}^2<\infty$, hence $\norm{d_k}\to0$ and $g_\alpha^{\mathrm{EP}}(x_k)\to0$, and every
  accumulation point of $\{x_k\}$ solves \ref{eq:ep}.
  \item If moreover $G$ is strongly monotone with modulus $c$ (Definition~\ref{def:bifun}) and
  $\alpha>\mu/(2c)$, then
  \[
    g_\alpha^{\mathrm{EP}}(x)\ \ge\ \big(c-\tfrac{\mu}{2\alpha}\big)\dist(x,\Sol)^2\qquad\forall\,x\in K,
  \]
  so $x_k\to x^*$, the unique solution of \ref{eq:ep}.
\end{enumerate}
\end{theorem}
\begin{proof}
By Proposition~\ref{prop:epgap}, $g_\alpha^{\mathrm{EP}}\in C^1$, and by hypothesis $\nabla g_\alpha^{\mathrm{EP}}$
is Lipschitz on $\mathcal C_0$ and the direction condition holds,
$\inner{\nabla g_\alpha^{\mathrm{EP}}(x_k)}{d_k}\le-\kappa\norm{d_k}^2$. Lemma~\ref{lem:descent} therefore
applies with merit $g_\alpha^{\mathrm{EP}}$ and direction constant $\kappa$
(Remark~\ref{rem:descent-generic}): the line search returns $\lambda_k\ge\lambda_{\min}$, the reference
$\{\mathcal T_k\}$ is nonincreasing, and
\[
  \mathcal T_k-\mathcal T_{k+1}\ \ge\ (1-\eta_{\max})\,\delta_1\kappa\lambda_{\min}\,\norm{d_k}^2 .
\]
Summing over $k=0,\dots,N$ and using $\mathcal T_{N+1}\ge0$,
\[
  (1-\eta_{\max})\,\delta_1\kappa\lambda_{\min}\sum_{k=0}^{N}\norm{d_k}^2\ \le\ \mathcal T_0-\mathcal T_{N+1}\ \le\ \mathcal T_0 ,
\]
so $\sum_k\norm{d_k}^2<\infty$ and $\norm{d_k}\to0$.

By coercivity (Assumption~\ref{as:standing}) the sublevel set $\{g_\alpha^{\mathrm{EP}}\le\mathcal T_0\}$ is
bounded, and $g_\alpha^{\mathrm{EP}}(x_k)\le\mathcal T_k\le\mathcal T_0$ keeps the iterates in it, so $\{x_k\}$
is bounded. Continuity of $\nabla_2 G$ then bounds $\norm{\nabla_2 G(x_k,x_k)}$, and \eqref{eq:ep-ub} gives
\[
  g_\alpha^{\mathrm{EP}}(x_k)\ \le\ \norm{\nabla_2 G(x_k,x_k)}\,\norm{d_k}\ \longrightarrow\ 0 .
\]
Since $y_\alpha$ is continuous (Berge's maximum theorem: the inner objective is jointly continuous by the standing hypotheses of this
subsection and strongly convex in $y$) and $y_\alpha(x)=x$ iff $x$ solves \ref{eq:ep}, any accumulation point
$\bar x=\lim_j x_{k_j}$ of the bounded sequence satisfies $y_\alpha(\bar x)-\bar x=\lim_j d_{k_j}=0$, hence
solves \ref{eq:ep}. This proves part~(i).

For part~(ii), apply Mastroeni's error bound~\cite[Prop.~4.2]{Mastroeni2003} to the regularizer
$H(x,y)=\tfrac1{2\alpha}\norm{x-y}_{\Dm^{-1}}^2$. It is symmetric, and $\nabla H$ is Lipschitz in each
argument with modulus $\mu/\alpha$, which is below $2c$ exactly when $\alpha>\mu/(2c)$. Under strong
monotonicity of $G$ with modulus $c$ it yields
\[
  g_\alpha^{\mathrm{EP}}(x)\ \ge\ \big(c-\tfrac{\mu}{2\alpha}\big)\dist(x,\Sol)^2\qquad\forall\,x\in K .
\]
With $g_\alpha^{\mathrm{EP}}(x_k)\to0$ from part~(i) and $c-\tfrac{\mu}{2\alpha}>0$,
\[
  \dist(x_k,\Sol)^2\ \le\ \frac{g_\alpha^{\mathrm{EP}}(x_k)}{\,c-\tfrac{\mu}{2\alpha}\,}\ \longrightarrow\ 0 ,
\]
so $x_k\to x^*$, the unique solution (uniqueness by strong monotonicity).
\end{proof}

\begin{remark}\label{rem:epdir}
The direction condition (Assumption~\ref{as:dir}) holds for any differentiable, strongly monotone $G$, the
equilibrium bifunctions that encode VI and MVI under strong monotonicity of $F$ being the special
case~\cite[Props.~3.2 and~5.2]{Mastroeni2003}. The scaled regularizer
$H(x,y)=\tfrac1{2\alpha}\norm{y-x}_{\Dm^{-1}}^2$ satisfies $\nabla_1 H+\nabla_2 H=0$, which reduces the
regularized direction condition to the unregularized one. The latter is supplied by strong monotonicity of
$G$ (Prop.~5.2), and Prop.~3.2 then transfers the descent property to $g_\alpha^{\mathrm{EP}}$.
\end{remark}

The following example illustrates Theorem~\ref{thm:ep} with two equilibrium problems.

\begin{example}[Equilibrium instances for Theorem~\ref{thm:ep}]\label{ex:ep}
\emph{(a) Encoding the $\mathrm{VI}$ of Example~\ref{ex:vi}.} With the same $F$, set
$G(x,y)=\inner{F(x)}{y-x}$. Then $G(x,x)=0$ and $G(x,\cdot)$ is affine, hence convex, and $G$ is
differentiable with $\nabla_1 G,\nabla_2 G$ continuous, so $g_\alpha^{\mathrm{EP}}\in C^1$
(Proposition~\ref{prop:epgap}). It is strongly monotone with modulus $c=m=1$,
\[
  G(x,y)+G(y,x)=-\inner{F(x)-F(y)}{x-y}=-m\norm{x-y}^2 .
\]
The inner minimization reduces to the negative Fukushima maximization, so $g_\alpha^{\mathrm{EP}}=g_\alpha$
and the direction condition (Assumption~\ref{as:dir}) holds with $\kappa=m$ by~\eqref{eq:dir-m}, as
anticipated in Remark~\ref{rem:epdir}. For Theorem~\ref{thm:ep}(ii), $\mu=1$ and any $\alpha>\mu/(2c)=\tfrac12$
(e.g.\ $\alpha=1$) give $c-\tfrac{\mu}{2\alpha}=\tfrac12>0$, hence $g_\alpha^{\mathrm{EP}}(x)\ge\tfrac12\dist(x,\Sol)^2$.
Since $g_\alpha^{\mathrm{EP}}=g_\alpha$, Proposition~\ref{prop:eb} already gives a VI error bound for every
$\alpha>0$, so the EP threshold $\alpha>\tfrac12$ is only the conservative general-EP estimate. It is disjoint
from the rate window $(0,0.4)$ of Remark~\ref{rem:exrate}, the small $\alpha$ certifying the R-linear rate and the
large one the EP error bound (Remark~\ref{rem:eprate}).

\emph{(b) A genuinely coupled equilibrium problem.} Let $G(x,y)=\inner{Bx+Ay-a}{y-x}$ on $K=[-1,1]^2$ with
$A=\tfrac12 I$, $B=\tfrac32 I+J_0$, $J_0=\begin{pmatrix}0&\tfrac32\\-\tfrac32&0\end{pmatrix}$,
$a=(0.3,-0.2)^\top$, and $m=1$. Then $G(x,x)=0$, the Hessian of $G(x,\cdot)$ is $A+A^\top=I\succ0$ (so
$G(x,\cdot)$ is convex, and $g_\alpha^{\mathrm{EP}}\in C^1$ since $G$ is a polynomial), and since $A\neq0$
couples $y$ the problem is a genuine equilibrium problem. A direct expansion gives $G(x,y)+G(y,x)=(y-x)^\top(A-B)(y-x)=-m\norm{x-y}^2$,
so $c=m=1$. The direction condition (Assumption~\ref{as:dir}) then holds by Remark~\ref{rem:epdir}, and with
$\Dm=I$ ($\mu=1$) and $\alpha>\tfrac12$ Theorem~\ref{thm:ep} yields global convergence and the error bound.
\end{example}

\begin{remark}[Toward an EP rate]\label{rem:eprate}
The error bound is the EP analogue of \eqref{eq:eb-explicit}. With the sufficient decrease it controls the
distance, but as for the regularized VI gap (Remark~\ref{rem:firstorder}) the EP gap is first order in the
residual by \eqref{eq:ep-ub}, so an R-linear rate is not automatic. Such a rate would follow once the
auxiliary-EP map is a contraction, available under strong monotonicity of $G$ with a Lipschitz-type
condition~\cite{QuocTran2008}. There is a tension: the error bound needs $\alpha$ large
($\alpha>\mu/(2c)$) whereas a contraction typically needs $\alpha$ small. The strongly-quasiconvex case, deriving a gap error bound from strong quasiconvexity of
$G(x,\cdot)$, is open beyond the proximal setting~\cite{Iusem2024,Lara2024}. Both are
left to future work.
\end{remark}

\begin{remark}[Scope and extensions]\label{rem:flagship}
Beyond Section~\ref{ss:ep}, the remaining extension is the pseudomonotone (residual) route: for merely
pseudomonotone $F$ or $G$ (or nonsmooth data) the descent identity \eqref{eq:gap-descent} fails and the
smooth-gap certificate (V1) is replaced by the residual/separating-hyperplane certificate
(V2)~\cite{Solodov1999}. A gap-PL route via the globally differentiable D-gap $g_{\alpha\beta}$
(Definition~\ref{def:dgap}) contracts the merit directly~\cite{Li2025,Zhao2025b}. The
equilibrium rate is discussed in Remark~\ref{rem:eprate}. A further direction is the strongly-quasiconvex
case---deriving a gap error bound from strong quasiconvexity of $G(x,\cdot)$ (cf.\ the negative
observation of Remark~\ref{rem:sq})---which would carry the R-linear rate to that class.
\end{remark}

% ============================================================================
\section{Numerical experiments}\label{sec:experiments}
% ============================================================================

This section has two aims. The first is to test the theory on controlled problems, where the constants
$m$, $L$, $\mu$ and a solution $x^*$ are known, so the predictions can be read off directly. The second is
to compare SGFM with standard methods for variational inequalities and equilibrium problems. Every
experiment confirms a proved statement.

\subsection{Methods}\label{ss:methods}

The proposed method is SGFM. The external baseline is the extragradient method~\cite{Korpelevich1976} for
variational inequalities, and Mastroeni's gap-descent scheme~\cite{Mastroeni2003} and the D-gap descent of
Bigi and Passacantando~\cite{Bigi2015} for equilibrium problems. Two internal ablations each disable a
single component of SGFM: \emph{SGFM-Eucl} turns off the scaling ($\Dm_k=I$), and \emph{SGFM-mono} replaces the
modified non-monotone line search by a monotone one ($\eta=0$). Table~\ref{tab:methods} collects them.

\begin{table}[H]
  \centering
  \caption{Methods compared in Section~\ref{sec:experiments}. The proposed SGFM method is benchmarked against
  standard external baselines (one for variational inequalities, two for equilibrium problems) and against
  internal ablations that disable a single component.}
  \label{tab:methods}
  \begin{tabularx}{\linewidth}{@{}llXl@{}}
    \toprule
    Method & Class & Update rule & Source \\
    \midrule
    \textbf{SGFM} (proposed) & VI/MVI/EP & scaled projection/resolvent with a modified non-monotone line search on $g_\alpha$ & this paper \\
    Extragradient           & VI        & two projections per step, fixed step $\alpha\in(0,1/L)$                              & \cite{Korpelevich1976} \\
    Gap descent             & EP        & Armijo descent on the regularized gap                                                & \cite{Mastroeni2003} \\
    D-gap descent           & EP        & descent on a family of D-gap functions                                               & \cite{Bigi2015} \\
    \addlinespace
    SGFM-Eucl                & VI/MVI/EP & SGFM with $\Dm_k=I$ (scaling disabled)                                                & this paper \\
    SGFM-mono                & VI/MVI/EP & SGFM with a monotone Armijo line search ($\eta=0$)                                    & this paper \\
    \bottomrule
  \end{tabularx}
\end{table}

\subsection{Experiments}\label{ss:plan}

Table~\ref{tab:experiments} lists the experiments. Each one targets a single proved statement, on the
problem where its constants are known. Experiments~1--6 test the theory; Experiment~7 is the comparison,
reported in two summary tables (one for variational inequalities, one for equilibrium problems) of
iterations, operator evaluations, and running time.

\begin{table}[H]
  \centering
  \caption{The experiments. Each confirms one result on a controlled problem.}
  \label{tab:experiments}
  \begin{tabularx}{\linewidth}{@{}lXl@{}}
    \toprule
    \# & Experiment (what it checks) & Confirms \\
    \midrule
    1 & Convergence from several starting points (from anywhere in $K$)        & Thms~\ref{thm:global},~\ref{thm:mvi},~\ref{thm:ep} \\
    2 & Residual decays in a straight line on a log scale, at rate $\theta$ & Thm~\ref{thm:rate} \\
    3 & Best-iterate residual tracks the $O(1/\sqrt N)$ floor as $m\downarrow0$  & Cor.~\ref{cor:sublinear} \\
    4 & Diagonal vs.\ identity scaling on an ill-conditioned operator           & Thm~\ref{thm:varmetric} \\
    5 & Modified vs.\ monotone line search (rate unchanged)                     & Thms~\ref{thm:global},~\ref{thm:rate} \\
    6 & Gap and D-gap vanish together and track each other                      & Prop.~\ref{prop:merit-equiv} \\
    7 & Comparison with the external baselines                                  & --- \\
    \bottomrule
  \end{tabularx}
\end{table}

The strongly monotone problems (Problems~\ref{prob:affine} and~\ref{prob:smep}, and the ill-conditioned
variant of Problem~\ref{prob:affine}) carry the rate and variable-metric results, where $m$, $L$, $\mu$ and
$x^*$ are all known. The merely monotone problems (Problem~\ref{prob:cournot} and Problem~\ref{prob:weak} at
$\rho\to0$) lie outside the strongly monotone theory and are used only for the weaker statements
(Experiments~3 and~6) and for the comparison. Because SGFM is a projection method, the comparison spans both the strongly monotone regime it is built
for and rotation-dominated operators, with the outcome reported in Section~\ref{ss:res-compare}.

\subsection{Test problems}\label{ss:testbed}

The testbed has six problems: four controlled synthetic instances, on which $m$, $L$, $\mu$ and a known
$x^*$ make the predictions checkable, and two standard problems from the literature. Each is stated in full,
so it can be reproduced from the data given here.

\begin{problem}[Affine, strongly monotone, asymmetric variational inequality]\label{prob:affine}
On the box $K=[-1,1]^n$, solve $\mathrm{VI}(F,K)$ for $F(x)=Mx+q$ with
\[
  M=D+\rho I+J ,\qquad D=\operatorname{diag}(d_1,\dots,d_n) ,\qquad \rho>0 ,\qquad J^\top=-J ,
\]
where $d_1,\dots,d_n$ are log-spaced between $1$ and a value that sets the conditioning of the symmetric
part ($10$ by default; $1000$ in the ill-conditioned variant of Experiment~4), and $J$ is bidiagonal with
$J_{i,i+1}=s=-J_{i+1,i}$, so $s$ sets the asymmetry. The symmetric part $\tfrac12(M+M^\top)=D+\rho I\succ0$
gives strong monotonicity with modulus $m=1+\rho$ and Lipschitz constant $L=\norm{M}_2$, while $J\neq0$
makes $F$ a genuine non-gradient operator. The vector $q=-Mx^*$ is fixed from an interior point
$x^*\in(-1,1)^n$, which is then the unique solution. Defaults: $n=100$, $\rho=1$, $s=2$. This is the
$n$-dimensional form of Example~\ref{ex:vi}, and it drives the rate, scaling, variable-metric, and
merit-equivalence experiments.
\end{problem}

\begin{problem}[Kojima--Shindo nonlinear complementarity problem~\cite{Kojima1986}]\label{prob:ncp}
Find $x\in\R^4$ with $x\geq0$, $F(x)\geq0$, and $\inner{x}{F(x)}=0$, where
\[
  F(x)=\begin{pmatrix}
    3x_1^2+2x_1x_2+2x_2^2+x_3+3x_4-6\\
    2x_1^2+x_1+x_2^2+10x_3+2x_4-2\\
    3x_1^2+x_1x_2+2x_2^2+2x_3+9x_4-9\\
    x_1^2+3x_2^2+2x_3+3x_4-3
  \end{pmatrix}.
\]
The solution $x^*=\big(\sqrt6/2,\,0,\,0,\,1/2\big)$ is degenerate, with $F(x^*)=\big(0,\,2+\sqrt6/2,\,0,\,0\big)$,
so that $x_3^*=0$ and $F_3(x^*)=0$. This standard nonlinear benchmark extends the testbed beyond the affine
setting.
\end{problem}

\begin{problem}[Weakly monotone family]\label{prob:weak}
On $K=[-1,1]^n$, solve $\mathrm{VI}(F,K)$ for the one-parameter family $F(x)=(\rho I+J)x+q$, with $J$ the
bidiagonal skew matrix of Problem~\ref{prob:affine} and small off-diagonal $s=0.5$. As $\rho\downarrow0$ the
operator passes from strongly monotone ($m=\rho>0$) to merely monotone ($\rho=0$, $F(x)=Jx+q$). The
asymmetry is kept modest: a large $s$ relative to $\rho$ makes $F$ rotation-dominated, which
projection and gap methods handle poorly (the extragradient method does not) and which would confound the
$m\to0$ study. An interior $x^*$ with $F(x^*)=0$ is fixed for every $\rho$, with $q=-(\rho I+J)x^*$. For
small $\rho>0$ the direction constant $\kappa=m>0$, so Corollary~\ref{cor:sublinear} applies and the
best-iterate residual tracks the $O(1/\sqrt N)$ floor as $\rho\downarrow0$.
\end{problem}

\begin{problem}[Mixed, $\ell_1$-regularized variational inequality]\label{prob:mvi}
On $K=[-1,1]^n$, solve the mixed problem for the affine operator $F(x)=Mx+q$ of Problem~\ref{prob:affine}
(defaults $n=100$, conditioning $10$, $s=2$) together with $\varphi(x)=\tau\norm{x}_1$, $\tau=0.5$, using the
scaled forward--backward step of Section~\ref{ss:mvi} with the fixed relaxation step. The $\ell_1$ term is
Lipschitz, so $\partial\varphi$ is bounded and the gap $g_\alpha^{\mathrm{MVI}}$ converges R-linearly
(Theorem~\ref{thm:mvi}); it shifts the solution off $-M^{-1}q$ and sparsifies it as in
Example~\ref{ex:mvi}, so no closed-form $x^*$ is used.
\end{problem}

\begin{problem}[Nash--Cournot electricity-market equilibrium~\cite{Quoc2012}]\label{prob:cournot}
A Nash--Cournot oligopoly of $n^c$ companies owning $n^g$ generating units in total (here $n^c=3$,
$n^g=6$) chooses unit outputs $x\in C^g=\{x\in\R^{n^g}:x_{\min}\le x\le x_{\max}\}$. With inverse demand
$p=378.4-2\sum_l x_l$ and smooth quadratic generation costs
$c_j(x_j)=\tfrac{\hat\alpha_j}{2}x_j^2+\hat\beta_j x_j+\hat\gamma_j$ (data in~\cite{Quoc2012}), the
Nikaido--Isoda reformulation is the equilibrium problem for
\[
  G(x,y)=\inner{A_1 x+B_1 y+a}{y-x}+c(y)-c(x),\qquad A_1=A+\tfrac32 B,\quad B_1=\tfrac12 B,
\]
with $c(x)=\sum_j c_j(x_j)$ and, writing $q^i\in\{0,1\}^{n^g}$ for the incidence vector of company $i$'s
units and $\tilde q^i=\mathbf 1-q^i$,
\[
  A=2\sum_{i=1}^{n^c}\tilde q^i(q^i)^\top,\qquad B=2\sum_{i=1}^{n^c}q^i(q^i)^\top,\qquad
  a=-378.4\sum_{i=1}^{n^c}q^i .
\]
The three companies own units $\{1\}$, $\{2,3\}$, $\{4,5,6\}$; the bounds are $x_{\min}=0$ and
$x_{\max}=(80,80,50,55,30,40)$; and the cost coefficients are
$\hat\alpha=(0.04,0.035,0.125,0.0116,0.05,0.05)$ and $\hat\beta=(2,1.75,1,3.25,3,3)$ (the constants
$\hat\gamma_j$ do not affect the equilibrium)~\cite{Quoc2012}. Because the units partition among the
companies, $A+B=2\,\mathbf 1\mathbf 1^\top\succeq0$, so $G$ is monotone but not strongly, and the quadratic
$c$ keeps $G(x,\cdot)\in C^1$; the equilibrium $x^*$ has no closed form. This economic instance
carries the EP global-convergence and comparison experiments, complementing the strongly monotone
Problem~\ref{prob:smep}.
\end{problem}

\begin{problem}[Synthetic strongly monotone equilibrium]\label{prob:smep}
On $K=[-1,1]^n$, solve the equilibrium problem for
\[
  G(x,y)=\inner{Bx+Ay-a}{y-x} ,\qquad A=\tfrac12 I ,\qquad B=\tfrac32 I+J_0 ,\qquad J_0^\top=-J_0 ,
\]
where $J_0$ is the bidiagonal skew matrix of Problem~\ref{prob:affine} with off-diagonal $s=1.5$ and
$a=(A+B)x^*$ is fixed from an interior $x^*$ (default $n=50$). Then $G(x,x)=0$, the Hessian
$\nabla^2_{yy}G=A+A^\top=I\succ0$, and $G(x,y)+G(y,x)=-\norm{x-y}^2$, so the bifunction modulus is $c=1$
(Definition~\ref{def:bifun}).
The solution $x^*=(A+B)^{-1}a$ and the modulus $c$ are known, so the error bound
$g_\alpha^{\mathrm{EP}}(x)\geq(c-\tfrac{\mu}{2\alpha})\dist(x,\Sol)^2$ and the threshold $\alpha>\mu/(2c)$ are
directly checkable. This is the $n$-dimensional form of Example~\ref{ex:ep}(b).
\end{problem}

\subsection{Setup}\label{ss:setup}

\subsubsection*{Implementation}\label{ss:setup-impl}
The methods are implemented in Julia~1.12.6. Each run uses a single thread on an Intel Core i9-9900K
(3.6\,GHz) with 32\,GB of memory, under Windows~11.

\subsubsection*{Performance measures}\label{ss:setup-meas}
For each run we record whether it converged, the number of iterations, the number of operator evaluations,
and the running time. The operator-evaluation count is the common measure of work: for SGFM it counts the
operator and inner-solve calls, and for the competitors the inner-subproblem solves, which dominate their
cost, so all methods are compared on the same axis. Along the run we record the residual $r$ defined below,
the gap $g_\alpha$, and the distance $\norm{x_k-x^*}$ when $x^*$ is known. Each problem uses a fixed random
seed, and all methods start from the same points, so the runs are reproducible.

\subsubsection*{Stopping rule}\label{ss:setup-stop}
All methods stop at the same accuracy. For a variational inequality we use the natural residual
\[
  r(x)=\norm{x-P_K\!\bigl(x-F(x)\bigr)} ,
\]
which is zero exactly at a solution. For an equilibrium problem we use the same residual with $F(x)$
replaced by $\nabla_2 G(x,x)$, the gradient of $G$ in its second argument at $y=x$; this matches $r$ when
the bifunction encodes a variational inequality. A run stops when $r(x_k)\le\varepsilon$ with
$\varepsilon=10^{-6}$, or at the iteration cap. Each method has its own internal residual, but reporting the
common $r$ makes ``converged'' mean the same accuracy for every method. The residual $r$ is zero only at a
solution and is, up to constants, equivalent to the gap $g_\alpha$ and to $\dist(x_k,\Sol)$ under the
standing assumptions (cf.\ Section~\ref{ss:meriteq}); stopping on it is therefore consistent with the gap-
and distance-based rate of Theorem~\ref{thm:rate}.

\subsubsection*{Iteration cap}\label{ss:setup-cap}
The cap follows from the rate. An R-linear residual $r_k\le Cq^k$ reaches $\varepsilon$ in about
$\log(C/\varepsilon)/\log(1/q)$ steps. We take a fixed multiple of this estimate with a conservative $q$,
so the cap does not cut off a method that would otherwise converge. The value used in each experiment is
stated with that experiment.

\subsubsection*{Parameters}\label{ss:setup-params}
The free parameters of each method (Table~\ref{tab:params}) are fixed once and reused in every experiment.
They come from a search on the test problems, run the same way for every method: a one-at-a-time sweep first
shows how sensitive each parameter is, then a Latin-hypercube search over the joint ranges selects the
setting. The score is the same for all methods --- first the fraction of runs that converge, then the number
of operator evaluations --- and the best setting is kept. Table~\ref{tab:params} reports that setting. For SGFM the line-search constants $\delta_1=10^{-3}$, $\delta_2=10^{-4}$ and the step cap
$\lambda_{\max}=1$ are kept at standard values, not searched.

\subsubsection*{Starting points}\label{ss:setup-starts}
Each problem is run from five fixed feasible starting points, the same five for every method: the lower and
upper corners of the box, its centre, a further corner alternating between the lower and upper bounds, and
one seeded interior point. The corners are the points farthest from the interior solution, so convergence
from them is the direct test of the global statements (Experiment~1).

\begin{table}[H]
  \centering
  \caption{Parameter values used in all experiments, fixed by the search of Section~\ref{ss:setup} and reported to the precision shown. Each row uses the notation of its method
  (Table~\ref{tab:methods}).}
  \label{tab:params}
  \begin{tabular}{@{}ll@{}}
    \toprule
    Method & Values \\
    \midrule
    SGFM (VI, EP)                        & $\alpha=0.195$,\quad $\eta=0.063$,\quad $\sigma=0.421$,\quad $\mu=5.88$,\quad diagonal $\Dm_k$ \\
    SGFM (MVI)                           & $\alpha=0.131$,\quad fixed step $\lambda=0.946$ \\
    Extragradient                       & $\alpha=0.916/L$ \\
    Gap descent (Mastroeni)             & $\rho=5.90$,\quad Armijo constant $3.9\times10^{-4}$ \\
    D-gap descent (Bigi--Passacantando) & $\gamma=0.3$,\quad $\delta=0.4$,\quad $\eta=0.9$ \\
    \bottomrule
  \end{tabular}
\end{table}

\subsection{Convergence and rate}\label{ss:res-rate}

The starting point changes neither whether the method converges nor how fast. Figure~\ref{fig:global} runs
SGFM on the affine variational inequality (Problem~\ref{prob:affine}) from the five points of
Section~\ref{ss:setup}, which include the box corners farthest from the interior solution. All five reach a
residual of $10^{-6}$ in about $120$ iterations, and after a short initial phase the five curves run
parallel. Parallel lines on a logarithmic scale mean a common contraction factor, so the rate is set by the
operator through $m$, $L$, $\mu$, and $\alpha$, while the start only raises or lowers the curve. This is the
global convergence of Theorem~\ref{thm:global}, and it matches the rate being independent of $x_0$ in
Theorem~\ref{thm:rate}.

The contraction is geometric and holds throughout the run, not only near the solution. On a logarithmic
scale the residual of a single run is a straight line from $10^{1}$ to $10^{-7}$ (Figure~\ref{fig:rate}), so
the residual drops by a fixed factor at every step across seven orders of magnitude. The one departure is a
short bend in the first few iterations, where the line search settles on its step, after which the slope is
constant. This is the R-linear rate of Theorem~\ref{thm:rate}.

How fast the method runs depends sharply on how strongly monotone the operator is. Figure~\ref{fig:sublinear}
runs SGFM on the weakly monotone family (Problem~\ref{prob:weak}), whose modulus is $m=\rho$, at $\rho=10^{-1}$,
$10^{-2}$, and $10^{-3}$. At $\rho=10^{-1}$ the operator is strongly monotone and the best residual reaches
$10^{-6}$ within the iteration budget. At $\rho=10^{-2}$ and $10^{-3}$ the method makes little headway over
the same budget and stays on the slow $O(1/\sqrt N)$ floor of Corollary~\ref{cor:sublinear}. The transition
is gradual, not abrupt, which is what the rate predicts: as $m\to0$ the contraction factor $q$ tends to one
and the R-linear rate passes continuously into the sublinear floor. The method is fast when the operator is
genuinely strongly monotone and slow when it is close to merely monotone.

\begin{figure}[H]
  \centering
  \includegraphics[width=0.7\linewidth]{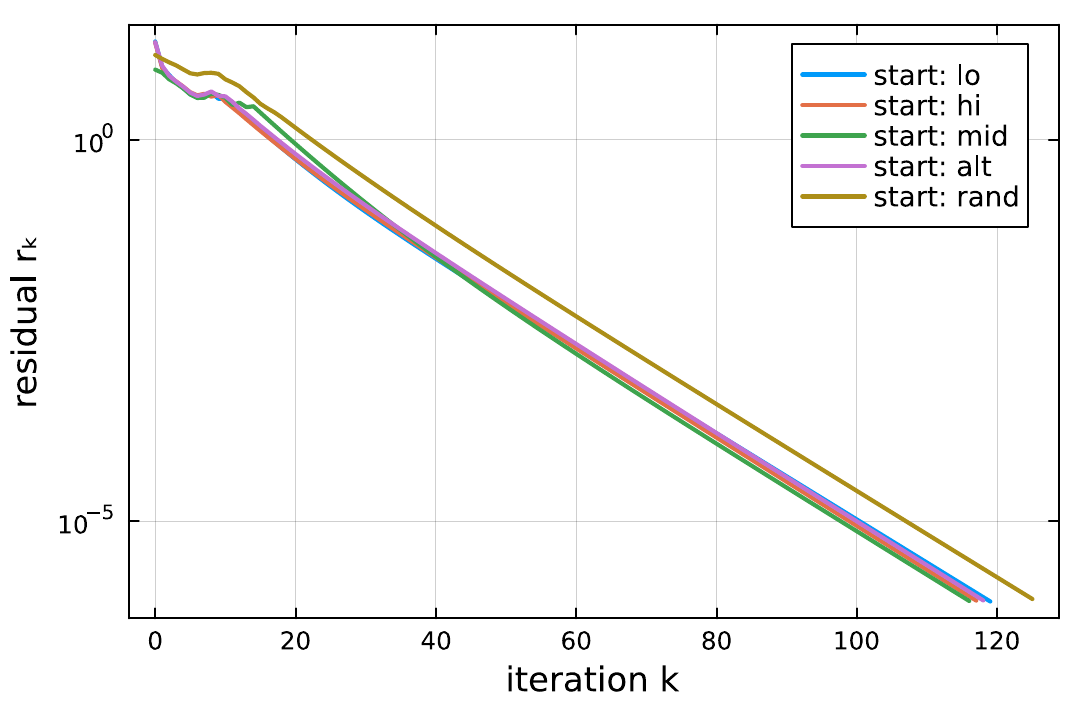}
  \caption{Global convergence. SGFM on the affine variational inequality (Problem~\ref{prob:affine}) from the
  five starting points of Section~\ref{ss:setup}. The residual reaches $10^{-6}$ from every start with a
  common asymptotic slope.}
  \label{fig:global}
\end{figure}

\begin{figure}[H]
  \centering
  \includegraphics[width=0.7\linewidth]{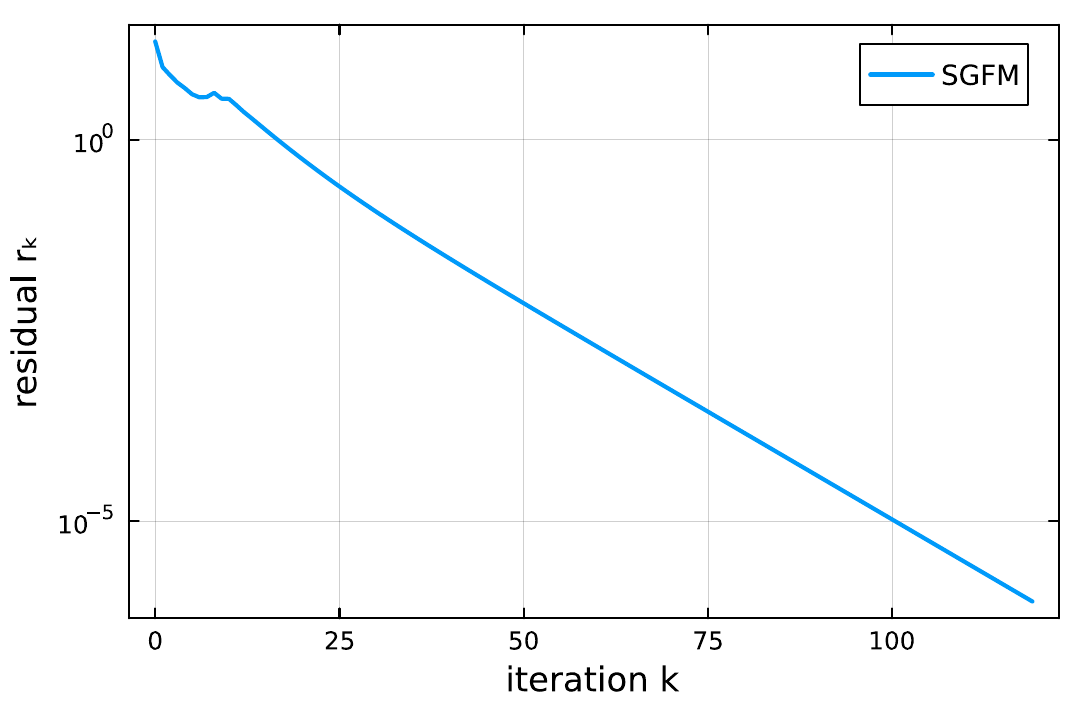}
  \caption{R-linear rate. The residual of SGFM on Problem~\ref{prob:affine} is a straight line on a
  logarithmic scale (Theorem~\ref{thm:rate}).}
  \label{fig:rate}
\end{figure}

\begin{figure}[H]
  \centering
  \includegraphics[width=0.7\linewidth]{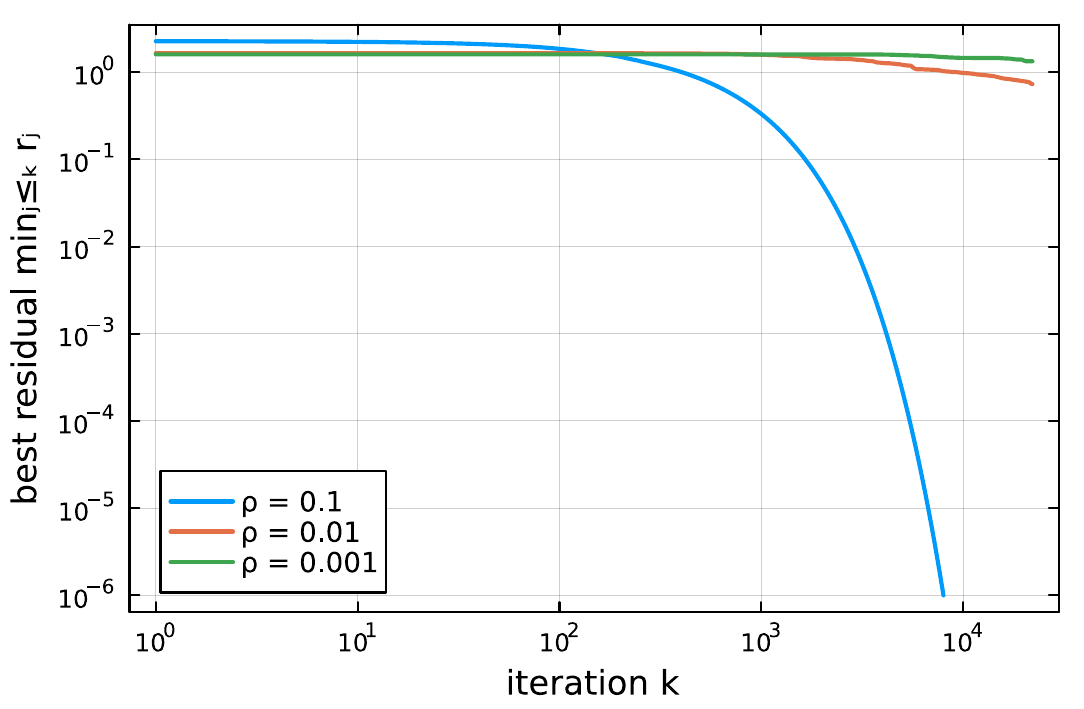}
  \caption{Best-iterate floor as the operator weakens. SGFM on the weakly monotone family
  (Problem~\ref{prob:weak}) at $\rho=10^{-1},10^{-2},10^{-3}$. The strongly monotone member ($\rho=10^{-1}$)
  converges, while smaller $\rho$ approach the $O(1/\sqrt N)$ floor of Corollary~\ref{cor:sublinear}.}
  \label{fig:sublinear}
\end{figure}

\subsection{The variable metric}\label{ss:res-metric}

On a badly conditioned operator the metric decides the speed, and the diagonal metric wins when the
conditioning lies on the diagonal. Figure~\ref{fig:scaling} runs SGFM with the diagonal $\Dm_k$ and with
$\Dm_k=I$ on the affine operator conditioned at $10^3$. Both residuals are straight lines, so both converge
R-linearly as Theorem~\ref{thm:varmetric} guarantees, but the diagonal metric reaches $10^{-6}$ in about
$150$ iterations against about $1300$ for the Euclidean step, almost ten times fewer. The diagonal metric
rescales each coordinate by its own curvature, so the method sees a much smaller effective conditioning.
The size of the gain therefore depends on where the ill-conditioning sits. A diagonal metric removes
conditioning carried on the diagonal of the operator and cannot touch conditioning carried by the
off-diagonal coupling. This is why it helps so much here, where the conditioning is diagonal by construction,
and why a projection method slows on problems whose difficulty is rotational instead.

\begin{figure}[H]
  \centering
  \includegraphics[width=0.7\linewidth]{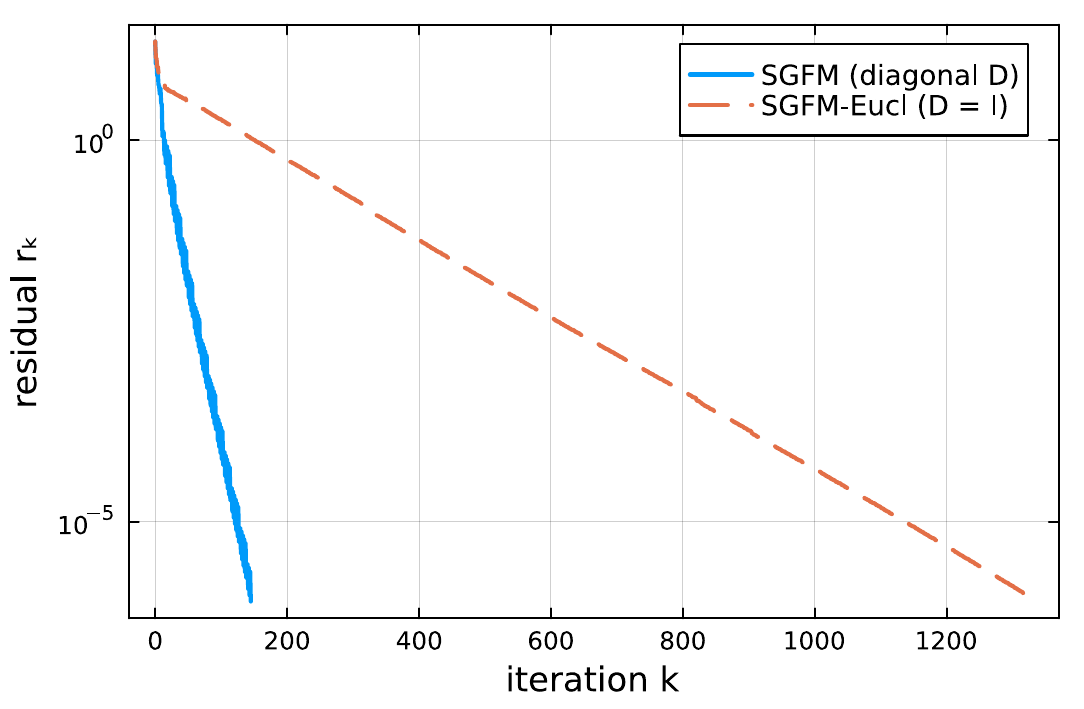}
  \caption{Variable metric. SGFM with the diagonal $\Dm_k$ and with $\Dm_k=I$ (SGFM-Eucl) on the affine
  operator conditioned at $10^3$. The diagonal metric is about an order of magnitude faster and keeps the
  R-linear rate of Theorem~\ref{thm:varmetric}.}
  \label{fig:scaling}
\end{figure}

\subsection{The modified non-monotone line search}\label{ss:res-ls}

On these smooth strongly monotone problems the non-monotone line search neither speeds the method up nor
slows it down, as the construction predicts. Figure~\ref{fig:lsearch} runs SGFM and
its monotone version ($\eta=0$) on the affine operator, and the two residual curves lie on top of each other.
Figure~\ref{fig:envelope} shows the reason. The reference value $T_k$ stays equal to the gap $g_\alpha(x_k)$
because the gap already falls at every step, so the relaxation in the Armijo test is never triggered. The
non-monotone terms come into play only when a step would raise the gap, and a smooth strongly monotone
operator never forces that. The experiment thus measures the cost of the modification
and finds it negligible, leaving the relaxation available for the non-monotone problems it is designed for.

\begin{figure}[H]
  \centering
  \includegraphics[width=0.7\linewidth]{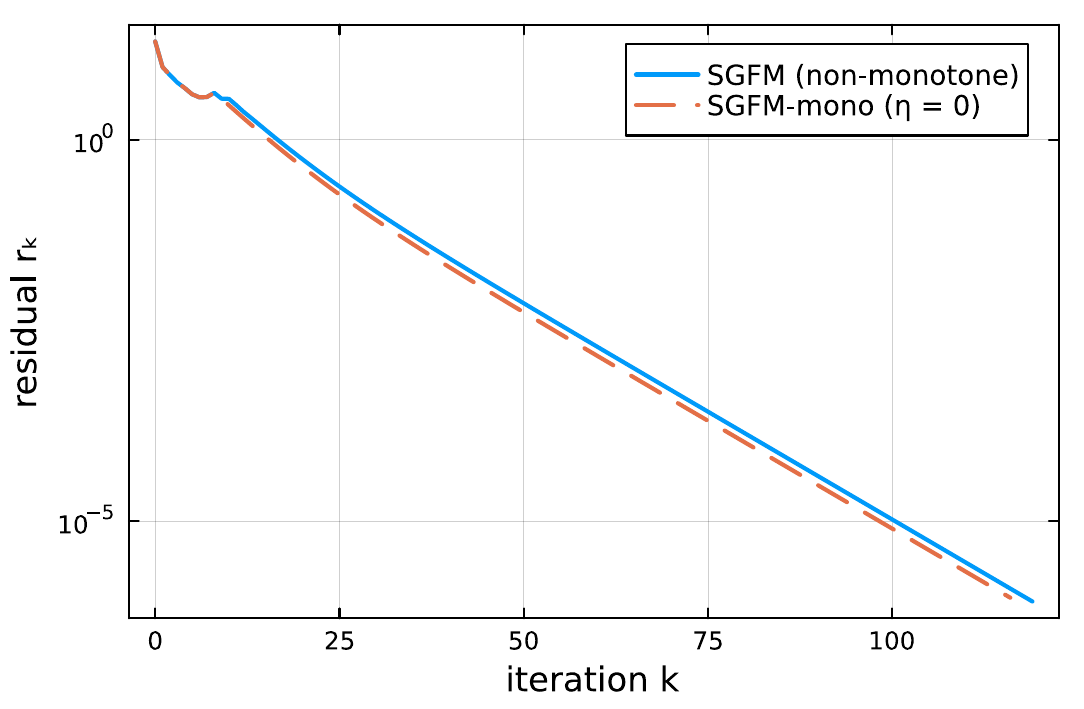}
  \caption{Modified versus monotone line search. SGFM and its monotone variant ($\eta=0$) on
  Problem~\ref{prob:affine}. The residual curves coincide, so the non-monotone terms do not change the rate.}
  \label{fig:lsearch}
\end{figure}

\begin{figure}[H]
  \centering
  \includegraphics[width=0.7\linewidth]{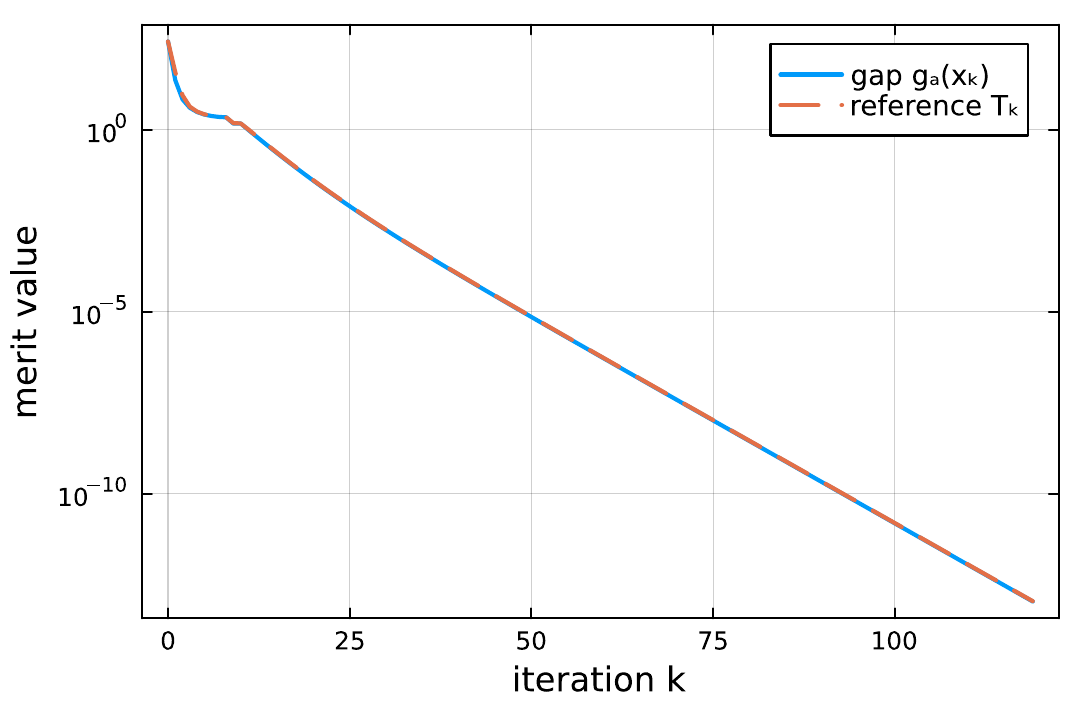}
  \caption{The moving-average reference $T_k$ against the gap $g_\alpha(x_k)$ for SGFM on
  Problem~\ref{prob:affine}. The gap decreases monotonically, so $T_k$ tracks it and the relaxation stays
  dormant.}
  \label{fig:envelope}
\end{figure}

\subsection{Merit equivalence}\label{ss:res-merit}

The gap and the D-gap measure progress in the same way. Figure~\ref{fig:merit} runs SGFM, which reduces the
gap $g_\alpha$, and the method of Bigi and Passacantando, which reduces the D-gap, on the strongly monotone
equilibrium problem (Problem~\ref{prob:smep}). Both merit values fall to zero as straight lines, so both
vanish R-linearly. The two slopes differ because the gap and the D-gap agree only up to constant
factors (Proposition~\ref{prop:merit-equiv}), not because one method is faster than the other.
The consequence for the comparison is that the gap we stop on and the D-gap the competitor stops on track the
same quantity, so the two methods are measured on common ground.

\begin{figure}[H]
  \centering
  \includegraphics[width=0.7\linewidth]{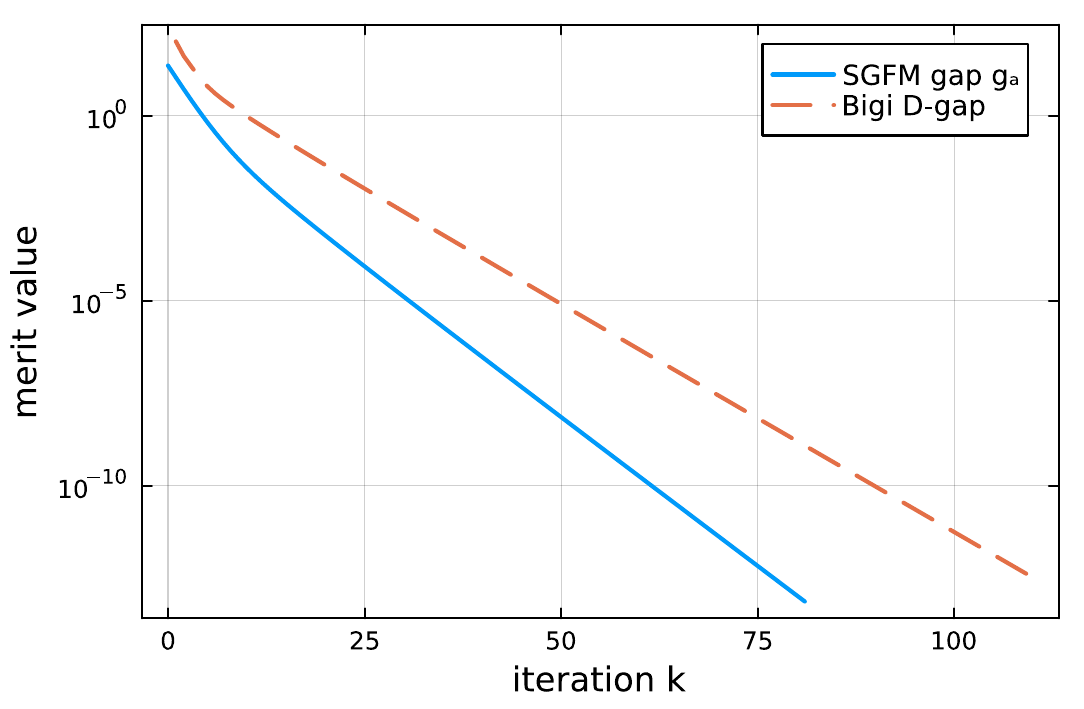}
  \caption{Merit equivalence. The gap $g_\alpha$ along the SGFM iterates and the D-gap along the
  Bigi--Passacantando iterates on the equilibrium problem (Problem~\ref{prob:smep}). Both vanish R-linearly,
  consistent with Proposition~\ref{prop:merit-equiv}.}
  \label{fig:merit}
\end{figure}

\subsection{Comparison with standard methods}\label{ss:res-compare}

SGFM is on par with the specialized methods on the equilibrium and mixed problems and is slower than the
extragradient method on the variational ones. Tables~\ref{tab:compare-vi} and~\ref{tab:compare-ep} give the
median over the five starts.

On the variational problems (Table~\ref{tab:compare-vi}) the extragradient method is the faster solver, and
the gap between the two grows as the operator weakens. On the well-conditioned strongly monotone instance the
counts are close, $94$ iterations for the extragradient method against $118$ for SGFM. On the weakly monotone
instance the extragradient method needs $157$ iterations and SGFM needs about $8000$, because the extragradient
step stays effective as the modulus $m$ shrinks while the SGFM rate falls off with $m$, as
Figure~\ref{fig:sublinear} already showed. On the non-monotone Kojima--Shindo problem SGFM does not reach the
tolerance, while the extragradient method reaches one of the problem's several solutions. Both results follow
the theory. SGFM rests on a strong-monotonicity contraction, and the extragradient method is built for plain
monotone and some non-monotone operators.

On the equilibrium problems (Table~\ref{tab:compare-ep}) SGFM is even with the specialized methods where its
theory applies. On the strongly monotone instance SGFM and Mastroeni's gap descent both converge in $81$
iterations and the D-gap method in $110$. Mastroeni's method uses fewer operator evaluations, $163$ against
$326$, because SGFM does more work per step in its inner solve and line search, so equal iteration counts do
not mean equal cost. On the merely monotone Nash--Cournot instance only the D-gap method converges, the
regime it was designed for. On the mixed problem SGFM solves the $\ell_1$-regularized instance in $73$
iterations using the fixed relaxation step in place of the line search.

SGFM is not the fastest method on any single class, and the comparison does not aim to make it one. Each
competitor is specialized to a single class and tuned for it. What the experiments show is that one method,
carrying one analysis, gives the proved guarantees across all three classes: global convergence and the
R-linear rate under strong monotonicity, the same rate under the variable metric, and the gap error bound.
The price of that reach is the per-iteration work of the gap-based step, and on the variational problems a
higher iteration count than a method built for them alone.

\input{imgs/tab_compare_vi}
\input{imgs/tab_compare_ep}

% ============================================================================
\section{Conclusion}\label{sec:conclusion}
% ============================================================================

We extended the scaled gradient modified non-monotone line search method of Ansari et al.~\cite{Ansari2026} from optimization to
variational inequalities, mixed variational inequalities, and equilibrium problems, through a gap-function
formulation in which the scaled projection coincides with the gap-defining maximizer. For variational and
mixed variational inequalities with a strongly monotone Lipschitz operator the analysis yields global
convergence and an R-linear rate via a gap error bound, under a fixed or a controlled-change variable
metric. For equilibrium problems it yields global convergence and an error bound, with the R-linear rate
reduced to such a bound and left open. The numerical experiments confirm these guarantees on controlled
problems. The residual contracts R-linearly under strong monotonicity, the diagonal metric is about an order
of magnitude faster on ill-conditioned operators, and the method converges from every start across all three
classes. Against specialized methods SGFM is competitive on the equilibrium and mixed problems and slower than
the extragradient method on the variational ones, the price of a single framework that covers all three.

Several questions are left open, each flagged in the analysis. The first is the equilibrium rate. The
Mastroeni-gap error bound of Theorem~\ref{thm:ep} controls the distance to the solution but is only first
order in the residual, so a geometric rate is not automatic. It would follow from a contraction of the
auxiliary-equilibrium map under a Lipschitz-type condition on $G$, once the conflicting requirements on
$\alpha$ are reconciled---the error bound needs $\alpha$ large, a contraction typically small
(Remark~\ref{rem:eprate}). The second is the pseudomonotone and nonsmooth range. When the operator is merely
pseudomonotone or the data nonsmooth, the smooth-gap descent identity fails, and the line search moves from
the gap certificate (V1) to a separating-hyperplane test on the natural residual (V2), carrying the
non-monotone reference onto the residual rather than the gap (Remark~\ref{rem:flagship}). The third is the
strongly quasiconvex class of the base method: because the gap does not inherit strong quasiconvexity
(Remark~\ref{rem:sq}), an R-linear rate there must come from a gap error bound derived directly from strong
quasiconvexity of $G(x,\cdot)$. Extensions to quasi-variational inequalities and to a Hilbert-space setting,
where the contraction and error-bound arguments use no compactness, are longer-term.

%% === BACK MATTER ===

\section*{Declarations}

 \textbf{Conflict of interest:} The author declares that he has no conflict of interest.

\section*{Data Availability}
The Julia code implementing the method, the competitor solvers, the test problems of
Section~\ref{ss:testbed}, and the scripts that generate the tables and figures of
Section~\ref{sec:experiments}, together with the generated benchmark data, are available from the author on
reasonable request and will be released in a public repository upon acceptance. Because every test problem
is specified in full in the manuscript, all experiments are reproducible from the stated data.

\section*{Funding}
This research did not receive any specific grant from funding agencies in the public, commercial, or
not-for-profit sectors. (Funding not applicable).

\section*{AI Use Declaration}
During the preparation of this work the author used Claude (Anthropic) to assist with
manuscript editing, including tightening prose, verifying LaTeX formatting, and checking internal
consistency of cross-references and notation. After using this tool, the author reviewed and edited
the content as needed and takes full responsibility for the content of the publication.

\section*{Acknowledgment}
The author gratefully acknowledges the support provided by King Fahd University of Petroleum \& Minerals
(KFUPM) and its Interdisciplinary Research Center (IRC) for Smart Mobility and Logistics.

\printbibliography
\end{document}

%% file: imgs/tab_compare_vi.tex
% Auto-generated by jcode/scripts/s70_figures_tables.jl — do not edit by hand.
% Preamble needs: \usepackage{tabularx}, booktabs, array, multirow.
\begin{table}[H]
  \centering
  \caption{Comparison on the variational and mixed problems (median over five starts).}
  \label{tab:compare-vi}
  \begin{tabularx}{\linewidth}{@{}ll*{4}{>{\raggedleft\arraybackslash}X}@{}}
    \toprule
    & & & \multicolumn{3}{c}{Median over converged starts} \\
    \cmidrule(lr){4-6}
    Problem & Method & Conv. & Iter. & Op.\ evals & $\norm{x-x^*}$ \\
    \midrule
    \multirow{2}{*}{Problem~\ref{prob:affine}} & SGFM      & 5/5 & 118 & 346 & $2.0\!\times\!10^{-7}$ \\
                         & EG        & 5/5 & 94 & 189 & $2.8\!\times\!10^{-7}$ \\
    \addlinespace
    \multirow{2}{*}{Problem~\ref{prob:ncp}} & SGFM      & 0/5 & --- & --- & --- \\
                         & EG        & 5/5 & 292 & 585 & 3.05 \\
    \addlinespace
    \multirow{2}{*}{Problem~\ref{prob:weak}} & SGFM      & 5/5 & 8006 & 16013 & $1.6\!\times\!10^{-6}$ \\
                         & EG        & 5/5 & 157 & 315 & $9.3\!\times\!10^{-6}$ \\
    \addlinespace
    \multirow{1}{*}{Problem~\ref{prob:mvi}} & SGFM-MVI  & 5/5 & 73 & 74 & --- \\
    \bottomrule
  \end{tabularx}
  \smallskip\noindent{\footnotesize Problem~\ref{prob:ncp} is multi-solution, and the extragradient method converges to a solution other than the reported $x^*$.}
\end{table}

%% file: imgs/tab_compare_ep.tex
% Auto-generated by jcode/scripts/s70_figures_tables.jl — do not edit by hand.
% Preamble needs: \usepackage{tabularx}, booktabs, array, multirow.
\begin{table}[H]
  \centering
  \caption{Comparison on the equilibrium problems (median over five starts).}
  \label{tab:compare-ep}
  \begin{tabularx}{\linewidth}{@{}ll*{4}{>{\raggedleft\arraybackslash}X}@{}}
    \toprule
    & & & \multicolumn{3}{c}{Median over converged starts} \\
    \cmidrule(lr){4-6}
    Problem & Method & Conv. & Iter. & Op.\ evals & $\norm{x-x^*}$ \\
    \midrule
    \multirow{3}{*}{Problem~\ref{prob:cournot}} & SGFM      & 0/5 & --- & --- & --- \\
                         & MAS       & 0/5 & --- & --- & --- \\
                         & Bigi      & 5/5 & 220 & 1036 & --- \\
    \addlinespace
    \multirow{3}{*}{Problem~\ref{prob:smep}} & SGFM      & 5/5 & 81 & 326 & $2.6\!\times\!10^{-7}$ \\
                         & MAS       & 5/5 & 81 & 163 & $2.7\!\times\!10^{-7}$ \\
                         & Bigi      & 5/5 & 110 & 664 & $2.6\!\times\!10^{-7}$ \\
    \bottomrule
  \end{tabularx}
\end{table}